\documentclass[notitlepage]{amsart}

\usepackage[utf8]{inputenc}
 \usepackage{setspace}
\usepackage[cmtip,arrow]{xy}
\usepackage{lscape}
\usepackage{latexsym}
\usepackage[activeacute,english]{babel}
\usepackage{graphicx}
\usepackage{amsmath}
\usepackage{amsfonts}

\usepackage{color}

\definecolor{vermelho}{rgb}{0.9, 0.0, 0.0}
\usepackage{pst-node}

\usepackage{amsthm}
\usepackage{amssymb}
\usepackage{amscd}
\usepackage{mathrsfs}
\usepackage{tikz}
\usepackage{pb-diagram,pb-xy}
\xyoption{all}
\usepackage{graphics}
\usepackage{stmaryrd}
\usepackage{yfonts,bbm}

\newtheorem{theorem}{Theorem}[section]
\newtheorem{lemma}[theorem]{Lemma}
\newtheorem{proposition}[theorem]{Proposition}

\newtheorem{definition}[theorem]{Definition}

\newtheorem{Remark}[theorem]{Remark}

\newcommand{\Supp}{\mathrm{Supp}}
\newcommand{\Hom}{\mathrm{Hom}}
\newcommand{\End}{\mathrm{End}}

\newcommand{\Ext}{\mathrm{Ext}}
\newcommand{\spl}{\mathrm{spl}}
	
\newcommand{\C}{\mathcal{C}}
\newcommand{\A}{\mathcal{A}}
\newcommand{\B}{\mathcal{B}}
\newcommand{\X}{\mathcal{X}}
\newcommand{\Y}{\mathcal{Y}}
\newcommand{\Z}{\mathcal{Z}}

\newcommand{\K}{\mathbb{K}}
\newcommand{\Ab}{\mathrm{Ab}}
\newcommand{\fl}{f.\ell.}
\newcommand{\T}{\mathcal{T}}
\newcommand{\F}{\mathcal{F}}
\newcommand{\Mod}{\mathrm{Mod}}
\newcommand{\MOD}{\mathrm{MOD}}
\newcommand{\modu}{\mathrm{mod}}
\newcommand{\proj}{\mathrm{proj}}
\newcommand{\pd}{\mathrm{pd}}
\newcommand{\gldim}{\mathrm{gl.dim.}}
\newcommand{\findim}{\mathrm{fin.dim.}}
\newcommand{\add}{\mathrm{add}}
\newcommand{\finp}{\mathrm{fin.p.}}
\newcommand{\ind}{\mathrm{ind}}

\newcommand{\Tr}{\mathrm{Tr}}

\begin{document}
\title[Standardly stratified lower triangular rings]{Standardly stratified lower triangular $\K$-algebras with enough idempotents}

\author{E. Marcos, O. Mendoza,  C. S\'aenz and  V.  Santiago}
\thanks{
The authors thanks project PAPIIT-Universidad Nacional Aut\'onoma de M\'exico  IN100520. The first mentioned author was supported by the thematic project of FAPESP 2014/09310-5, and research grant from CNPq 302003/2018-5  }
\dedicatory{}
\maketitle

\begin{abstract}
In this paper we study the lower triangular matrix $\K$-algebra $\Lambda:=\left[\begin{smallmatrix}
T & 0 \\ 
M & U
\end{smallmatrix}\right],$ where  $U$ and $T$ are basic $\K$-algebras with enough idempotents and $M$ is an $U$-$T$-bimodule where $\K$ acts centrally. Moreover, we characterise in terms of 
 $U,$  $T$ and $M$ when, on one hand, the lower triangular matrix $\K$-algebra $\Lambda$ is standardly stratified in the sense of \cite{MOSS}; and on another hand, when $\Lambda$ is locally bounded in the sense of \cite{LidiaJose}. Finally, it is also studied several properties relating the projective dimensions in the categories of finitely generated modules $\mathrm{mod}(U)$, $\mathrm{mod}(T)$ and $\mathrm{mod}(\Lambda).$
\end{abstract}

\section{introduction}

The definitions of quasi-hereditary algebra and of standardly stratified algebra have been introduced in the context of Artin algebras or, more 
 broadly, for semi-primary rings. Note that in these cases, all the rings involved have unity and a finite number (up to isomorphisms) of simple modules. Quasi-hereditary algebras were introduce in \cite{CPS1} to deal with certain categories in the representation theory of Lie algebras and algebraic groups. They appear also in knot theory \cite{Xi}. To see more details about  standardly stratified algebras, we recommend the reader to see in \cite{ADL}, \cite{CPS2}, \cite{D}, \cite{W} and \cite{Xi2}.
\

 In \cite{MOSS}, the authors introduce the notion of standardly stratified ringoid which is a generalization of the classical notion of standardly stratified algebra for semi-primary rings with unity. Moreover, they study the module category of standardly stratified ringoids (in particular, algebras with 
 enough idempotents) and show that, even though they have no unity and an infinite number  (up to isomorphisms) of simple modules, many classic results about the $\Delta$- filtered representations of standardly stratified algebras can be generalized 
 to representations of standardly stratified ringoids. It is also remarked that certain equivalent characterizations of standardly stratified algebras and 
  quasi-hereditary algebras are not necessarily equivalent any more in the realm of ringoids. The theory developed in \cite{MOSS} is then specified to algebras (over commutative rings) which have enough idempotents but need not have an unity. As a result, the notions of standardly stratified algebra and quasi-hereditary algebra are expanded in \cite{MOSS} to algebras which may not have an unity.
\

In this paper we study the lower triangular matrix $\K$-algebra $\Lambda:=\left[\begin{smallmatrix}
T & 0 \\ 
M & U
\end{smallmatrix}\right],$ where  $U$ and $T$ are basic $\K$-algebras with enough idempotents and $M$ is an $U$-$T$-bimodule where $\K$ acts centrally. The motivation, for doing so, is to 
 generalize to the setting of \cite{MOSS} the results obtained by E. Marcos, H. Merklen and C. S\'aenz in \cite{MMS}, where they studied the lower triangular matrix algebra 
 $\Lambda$ for $U$ and $T$ finite dimensional left standardly stratified algebras with unity. It is worth mentioning that some of the results in \cite{MMS} where also obtained, independently, by B. Zhu for  the matrix algebra $\Lambda$ where $U$ and $T$ are quasi-hereditary algebras with unity \cite{Bin}.
\

Throughout the paper, the term ring (algebra) means associative ring (algebra)
which does not necessarily has a unity. We also fix a commutative ring $\K$ with unity and denote by $\fl(\K)$ the category of all the $\K$-modules having finite length. Following \cite{MOSS}, a $\K$-algebra with enough idempotents (w.e.i
$\K$-algebra) is a pair $(R; \{e_i\}_{i\in I}),$ where R is a K-algebra and 
$\{e_i\}_{i\in I}$ is a
family of orthogonal idempotents of $R$ such that 
$R=\oplus_{i\in I}\,e_iR=\oplus_{i\in I}\,Re_i.$
For such an algebra $R,$ we denote 
by $\Mod\,(R)$ the category of unitary left $R$-modules and by $\modu\,(R)$ the 
full subcategory of finitely generated left $R$-modules.
\

 It is said that a w.e.i. 
$K$-algebra $(R; \{e_i\}_{i\in I})$
is basic if it is Hom-finite,
 each $e_i$ is primitive and the elements of the set $\{Re_i\}_{i\in I}$ are pairwise non-isomorphic (for more details, see in Section 2). For such basic algebra $R,$ each partition $\tilde{\A}$ of the set $I$ induces (see Section 2) the set 
${}_{\tilde{\A}}\Delta$ of $\tilde{\A}$-standard modules. Thus, by Proposition \ref{CriterioSS}, we have that  $(R,\tilde{\A})$ is a left standardly stratified $\K$-algebra (in the sense of \cite{MOSS}) if, and only if, each $Re_i$ has a finite 
${}_{\tilde{\A}}\Delta$-filtration. We denote by $\F_f({}_{\tilde{\A}}\Delta)$ the class of all the left unitary $R$-modules which have a finite ${}_{\tilde{\A}}\Delta$-filtration.

\

Let us fix $(U,\{e_{i}\}_{i\in I})$ and $(T,\{f_{j}\}_{j\in J})$ basic w.e.i. $\K$-algebras, an $U$-$T$-bimodule $M$ (acting centrally on $\K$)
such that 
$\{e_iMf_j\}_{(i,j)\in I\times J}\subseteq \fl(\K),$ $\tilde{\A}=\{\tilde{\A}_i\}_{i<\alpha}$ a partition of $I,$ $\tilde{\B}=\{\tilde{\B}_j\}_{i<\beta}$ a partition of $J$ and $\tilde{\C}:=\tilde{\A}\vee\tilde{\B}$ a partition of the disjoint union $I\vee J$  (see Lemma \ref{JSLema}). By Lemma \ref{LBasic1}, we have that 
$(\Lambda,\{g_r\}_{I\vee J})$ is a basic w.e.i. $\K$-algebra, where $\Lambda$ is the  lower triangular matrix $K$-algebra 
$\left[\begin{smallmatrix} T  & 0 \\ M &  U\end{smallmatrix}\right].$
\

The first main result of this paper is the following one, see the details in Theorem \ref{TP1}.
\

{\bf Theorem A} For $\Lambda=\left[\begin{smallmatrix} T  & 0 \\ M &  U\end{smallmatrix}\right],$ the following statements are equivalent. \\
(a) $(\Lambda, \tilde{\C})$ is a left standardly stratified $\K$-algebra.\\
(b) $\{Mf_j\}_{j\in J}\subseteq\F_f({}_{\tilde{\A}}\Delta)$ and the pairs $(T, \tilde{\B})$ and $(U, \tilde{\A})$ are left standardly stratified $\K$-algebras.

In \cite{LidiaJose}, L. Angeleri H\"ugel and J. A. de la Pe\~na introduced locally bounded $\K$-algebras and studied their category of finitely generated modules. We use the developed theory in \cite{LidiaJose} to study $\modu\,(\Lambda)$ in terms of 
$\modu\,(T)$ and $\modu\,(U).$ In Section 3, we give enough conditions to have that $\Lambda,$ $T$ and $U$ are locally bounded, for more details see Proposition
\ref{Lblta}.
\

In order to state the next two main results of this paper, we introduce the following 
 notation. Let $(R,\{e_i\}_{i\in I})$ be a w.e.i. $\K$-algebra. We denote by 
 $\mathcal{P}^{<\infty}_R$ the class of all the $N\in\modu\,(R)$ with $\pd(N)<\infty,$ where $\pd(N)$ stands for the projective dimension of $N.$  For a class $\C\subseteq\Mod\,(R),$ we set 
 $\pd\,(\C):=\sup\{\pd(C)\;:\; C\in \C\}.$ The small global dimension of the ring $R$ is $\gldim(R):=\pd\,(\modu\,(R))$ and the small finitistic dimension of the ring $R$ is $\findim(R):=\pd\,(\mathcal{P}^{<\infty}_R).$
\

The second main result of this paper is the following one, see the details in Theorem \ref{MainT2}.
\

{\bf Theorem B} For $\Lambda=\left[\begin{smallmatrix} T &  0 \\ M & U \end{smallmatrix}\right]$ support finite, the following statements hold true.\\
(a) Let $L=(Y,\varphi,X)\in \Mod\,(\Lambda)$ and $\pd(Mf_j)<\infty$ $\forall\,j\in J.$ Then
$$L\in\mathcal{P}^{<\infty}_\Lambda\;\Leftrightarrow\; Y\in\mathcal{P}^{<\infty}_T\;\text{ and }\,X\in\mathcal{P}^{<\infty}_U.$$
(b) Let $m:=\max\,\{\pd(Mf_j)\}_{j\in J}<\infty,$ $a:=\findim(T),$ $\alpha:=\gldim(T),$ $b:=\findim(U)$ and $\beta:=\gldim(U).$ Then
\begin{equation*}
\begin{split}
\max\,\{\beta-m,\alpha\}\leq\gldim(\Lambda) & \leq\max\,\{\beta,\alpha+1+m\},\\
\max\,\{b-m,a\}\leq\findim(\Lambda) & \leq\max\,\{b,a+1+m\}.
\end{split}
\end{equation*}

The third main result of this paper is the following one, see the details in Theorem \ref{MainT3}.
\

{\bf Theorem C} For $\Lambda=\left[\begin{smallmatrix} T &  0 \\ M & U \end{smallmatrix}\right]$ support finite and $\pd(Mf_j)<\infty$ $\forall\,j\in J,$ the following statements are equivalent.\\
(a) $\F_f({}_{\tilde{\C}}\Delta)= \mathcal{P}_{\Lambda}^{<\infty}.$\\
(b) $\F_f({}_{\tilde{\A}}\Delta)= \mathcal{P}_{U}^{<\infty}$ and $\F_f({}_{\tilde{\B}}\Delta)= \mathcal{P}_{T}^{<\infty}.$\\
Moreover, if one of the above equivalent conditions holds true, then $\Lambda,$ $U$ and $T$ are locally bounded and left standardly stratified $\K$-algebras, 
$\findim(\Lambda)=\pd({}_{\tilde{\C}}\Delta),$ $\findim(T)=\pd({}_{\tilde{\B}}\Delta)$ and 
$\findim(U)=\pd({}_{\tilde{\A}}\Delta).$

\section{Preliminaries}

 In what follows, we recall some basic notations and results appearing in \cite{MOSS}.

{\sc Functor categories and ringoids.} Let $\C$ be a category. It is said that $\C$ is a {\bf $\K$-category} if $\Hom_\C(X,Y)$ is a 
$\K$-module for any $(X,Y)\in\C^2$ and the composition of morphisms in $\C$ is $\mathbb{K}$-bilinear.  Following B. Mitchell in 
 \cite{Mitchell},  a {\bf $\K$-ringoid} is just a skeletally small $\mathbb{K}$-category. In case $\K=\mathbb{Z},$ we name it 
 {\bf ringoid} instead of $\mathbb{Z}$-ringoid. 
\


Let $\mathfrak{R}$ be a ringoid and $\Ab$ be the category of abelian groups. We denote by $\Mod\,(\mathfrak{R})$ the category of left $\mathfrak{R}$-modules whose objects are all the additive (covariant) functors $M:\mathfrak{R}\to \Ab$ and morphisms are the natural transformations of functors.  Note that $\Mod\,(\mathfrak{R})$ is abelian and  bicomplete  since $\Ab$ is so. Denote by $\mathfrak{R}^{op}$ the opposite category of $\mathfrak{R}.$ The category of right $\mathfrak{R}$-modules is by definition $\Mod\,(\mathfrak{R}^{op}).$ We denote by $\modu\,(\mathfrak{R})$ the full subcategory of $\Mod\,(\mathfrak{R})$ whose objects are all the finitely generated left $\mathfrak{R}$-modules; and by $\proj(\mathfrak{R})$ the class of all the finitely generated projective left $\mathfrak{R}$-modules.
\

Let $\mathfrak{R}$ be a $\K$-ringoid. Following \cite{MOSS},  We recall that $M\in\Mod\,(\mathfrak{R})$ is {\bf finitely presented} if there is an exact sequence $P_1\to P_0\to M\to 0$ of $\mathfrak{R}$-modules such that 
$P_0,P_1\in\proj(\mathfrak{R}).$ We denote by $\finp(\mathfrak{R})$ the full subcategory of $\Mod\,(\mathfrak{R})$ whose objects 
are all the finitely presented  $\mathfrak{R}$-modules. The $\K$-ringoid $\mathfrak{R}$ is {\bf thick} if it is an additive category whose idempotents split. It is said that $\mathfrak{R}$ is Krull-Schmidt ({\bf KS-ringoid}, for short) if $\mathfrak{R}$ is a Krull-Schmidt category (that is an additive category in which every non-zero object decomposes into a finite coproduct of objects having local endomorphism ring). We say that $\mathfrak{R}$ is {\bf Hom-finite} if the $\K$-module $\Hom_{\mathfrak{R}}(a,b)$ is of finite length, for all $(a,b)\in\mathfrak{R}^2.$ Finally, a $\K$-ringoid which is Hom-finite and Krull-Schmidt is called {\bf locally finite} $\K$-ringoid.

{\sc Algebras with enough idempotents.} Let $R$ be a $\K$-algebra such that $R^2=R.$ We denote by $\MOD(R)$ the category of all the left $R$-modules and 
let $\Mod(R)$ be the full subcategory of all the unitary left $R$-modules $M,$ where unitary means that $RM=M.$ The full subcategory of $\Mod(R)$ whose objects are the finitely generated left $R$-modules will be denoted by $\modu(R).$ The class of all the finitely generated projective objects in $\Mod(R)$ is 
denoted by $\proj(R).$
\

Following \cite{MOSS}, we recall that a $\K$-algebra  with enough idempotents ({\bf w.e.i $\K$-algebra}) is a pair $(R,\{e_i\}_{i\in I}),$ where $R$ is a $\K$-algebra and
$\{e_i\}_{i\in I}$ is a family of orthogonal idempotents of $R$ such that  $R=\oplus_{i\in I}\,e_i R=\oplus_{i\in I} R e_i.$ In this case, it can be shown that  $R^2=R$ and $R=\oplus_{(i,j)\in I^2}\,e_i R\,e_j.$ A w.e.i $\K$-algebra 
$(R,\{e_i\}_{i\in I})$ is {\bf Hom-finite} if
$\{e_j Re_i\}_{(i,j)\in I^2}\subseteq \fl(\K).$
\

Let $(R,\{e_i\}_{i\in I})$ be a w.e.i. $\K$-algebra and let $R^{op}$ be the opposite $\K$-algebra of $R.$ Note that 
$(R^{op},\{e_i\}_{i\in I})$ is also a w.e.i. $\K$-algebra and it is known as the opposite w.e.i $\K$-algebra of $(R,\{e_i\}_{i\in I}).$ 

\begin{definition} Let  $(R,\{e_i\}_{i\in I})$ and  $(T,\{f_j\}_{j\in j})$ be  w.e.i. $\K$-algebras. We say that $\phi:(R,\{e_i\}_{i\in I})\to (T,\{f_j\}_{j\in j})$ is a morphism of w.e.i. $\K$-algebras if $\phi:R\to T$ is a morphism of $\K$-algebras such that $\phi(\{e_i\}_{i\in I})\subseteq \{f_j\}_{j\in j}.$
\end{definition}

Let  $\phi:(R,\{e_i\}_{i\in I})\to (T,\{f_j\}_{j\in j})$ be a morphism of w.e.i. $\K$-algebras. Consider $M\in\Mod\,(T).$ Note that $M\in\MOD(R),$ 
via the action of $R$ on $M$ 
\begin{center}$r\cdot m:=\phi(r)m$ $\forall\,r\in R,$ $\forall\,m\in M.$ 
\end{center}
Unfortunately, $M$ does not necessarily belong to $\Mod(R).$ In order 
to construct the so called {\bf change of rings functor} 
$\phi_{*}:\Mod(T)\to \Mod(R),$ we need to do some adjustment in the above 
 construction. Indeed, for $M\in\Mod\,(T),$ we set 
 $\phi_{*}(M):=\sum_{i\in I}\phi(e_i)M;$ and for a morphism 
$f:M\to N$ in $\Mod(T),$ we define $\phi_{*}(f)$ as the restriction of $f$ on 
$\phi_{*}(M).$ 

\begin{proposition}\label{CRF} Let  $\phi:(R,\{e_i\}_{i\in I})\to (T,\{f_j\}_{j\in j})$ be a morphism of w.e.i. $\K$-algebras. Then, the 
following statements hold true.
\begin{itemize}
\item[(a)] The above correspondence $\phi_{*}:\Mod(T)\to \Mod(R)$ is a well defined $\K$-linear functor.
\item[(b)] For any $M\in\Mod\,(T),$ the correspondence
$$\varepsilon_M:\coprod_{i\in I}\Hom_T(T\phi(e_i),M)\to \phi_{*}(M),\quad (\theta_i)_{i\in I}\mapsto\sum_{i\in I}\theta_i(\phi(e_i))$$ 
is an isomorphism of $R$-modules which is functorial on $M.$
\item[(c)] $\phi_{*}:\Mod(T)\to \Mod(R)$ is an exact functor.
\end{itemize}

\end{proposition}
\begin{proof} (a) Let $M\in\Mod(T).$ We start by showing that the abelian group $\phi_{*}(M)$ is an $R$-submodule of $M.$ Indeed, consider 
$r=\sum_{i\in I}e_ir_i\in R=\sum_{i\in I}e_iR$ and 
$x=\sum_{j\in I}\phi(e_j)x_j\in \phi_{*}(M)=\sum_{j\in I}\phi(e_j)M.$ Then, 
$$r\cdot x=\phi(r)x=\phi(\sum_{i\in I}e_ir_i)\sum_{j\in I}\phi(e_j)x_j=
\sum_{i\in I}\phi(e_i)(\sum_{j\in I}\phi(r_ie_j)x_j)\in\phi_{*}(M)$$
since $\sum_{j\in I}\phi(r_ie_j)x_j\in M$ $\forall\,i.$
\

We check now that $R\phi_{*}(M)=\phi_{*}(M).$ Let $x=\sum_{i\in I}\phi(e_i)x_i\in \phi_{*}(M).$ Then, by using that $\phi(e_i)x_i\in\phi_{*}(M)$ $\forall\,i,$ we get that
$$x=\sum_{i\in I}\phi(e_i)x_i=\sum_{i\in I}\phi(e_i)(\phi(e_i)x_i)= 
\sum_{i\in I}e_i\cdot(\phi(e_i)x_i)\in R\phi_{*}(M).$$
Therefore, we proved that $\phi_{*}(M)\in\Mod(R).$
\

Let $f:M\to N$ be a morphism in $\Mod(T).$ We need to show that 
$f(\phi_{*}(M))\subseteq \phi_{*}(N).$ Let $x=\sum_{i\in I}\phi(e_i)x_i\in \phi_{*}(M).$ Then,
$$f(x)=\sum_{i\in I}f(e_i\cdot x_i)=\sum_{i\in I}e_i\cdot f(x_i)=
\sum_{i\in I}\phi(e_i)f(x_i)\in\phi_{*}(N).$$
Thus the correspondence $\phi_{*}:\Mod(T)\to \Mod(R)$ is well defined and we let the reader to check that it is a $\K$-linear functor.
\

(b) Let $M\in\Mod(T)$ and $i\in I.$ Then $$e_i\cdot \phi_{*}(M)=\sum_{i'\in I} \phi(e_i)\phi(e_{i'})M=\phi(e_i)M.$$
Thus, for each $i\in I,$  $\varepsilon^i_M:\Hom_T(T\phi(e_i),M)\to e_i\phi_{*}(M),\quad f\mapsto f(\phi(e_i)),$ is an isomorphism of $\K$-modules, which is functorial on $M,$ and   $\Hom_T(T\phi(e_i),M)$ is an $R$-module since $M\in\MOD(R).$ Since 
$\varepsilon_M=\coprod_{i\in I}\varepsilon^i_M$ and $\varepsilon_M$ is a morphism of $R$-modules, we get (b).
\

(c) It follows from (b) since $T\phi(e_i)$ is projective in $\Mod(T),$ for all $i\in I.$
\end{proof} 

\begin{definition}\cite[Remark 6.8]{MOSS} It is said that a w.e.i. $\K$-algebra $(R,\{e_i\}_{i\in I})$ is basic if it is Hom-finite, each $e_i$ is primitive and 
$Re_i\not\simeq Re_j$ for $e_i\neq e_j.$
\end{definition}

\begin{Remark}\label{Rbasic}  Let $(R,\{e_i\}_{i\in I})$ be a Hom-finite w.e.i. $\K$-algebra. Note that, from \cite[Lemma 6.3]{MOSS}, we have that $(R,\{e_i\}_{i\in I})$ is basic if, and only if, each $e_i$ is primitive and 
$e_i R\not\simeq e_j R$ for $e_i\neq e_j.$
\end{Remark}

The following result will be used frequently. For the convenience of the reader, we give a proof.

\begin{lemma}\label{primero}
Let $(R,\{e_{i}\}_{i\in I})$ be a w.e.i $\K$-algebra and $M\in \Mod\,(R)$. Then $M=\bigoplus_{i\in I}e_{i}M.$
\end{lemma}
\begin{proof} Since $RM=M$ and $R=\oplus_{i\in I}\,e_i R,$ we have that $M=\sum_{i\in I} e_{i}M.$ Let us show that 
$e_{i}M\cap \big (\sum_{j\neq i}e_{j}M\big)=0$ for $i\neq j.$ Indeed, let $x\in e_{i}M\cap \big (\sum_{j\neq i}e_{j}M\big)$. Then $x=e_{i}m=\sum_{j\in I-\{i\}}e_jm_j$ with $m, m_j\in M$ and hence
$x=e_{i}m=e_{i}^{2}m=\sum_{j\in I-\{i\}}e_ie_jm_j=0;$ proving the result.
\end{proof}

For a given w.e.i $\K$-algebra $(R,\{e_i\}_{i\in I}),$ we have its associated $\K$-ringoid $\mathfrak{R}(R)$ whose objects are 
$\{e_i\}_{i\in I}$ and the morphisms from $e_i$ to $e_j$ are given by the set $\Hom_{\mathfrak{R}(R)}(e_i,e_j):=e_j R e_i.$ The composition of morphisms in $\mathfrak{R}(R)$ is given by the multiplication of the ring $R.$ Note that 
$\mathfrak{R}(R)^{op}=\mathfrak{R}(R^{op}).$

We recall that $Y:\mathfrak{R}(\Lambda)\to \mathrm{Mod}(\mathfrak{R}(\Lambda))$ is the Yoneda's
contravariant functor, where $Y(e):=\mathrm{Hom}_{\mathfrak{R}(\Lambda)}(e,-)$ for any $e\in \mathfrak{R}.$
\

The following result is well-known in the mathematical folklore, see \cite[Proposition 6.1]{MOSS} for a proof. 

\begin{proposition}\label{WeiEquiv1} Let $(\Lambda,\{e_i\}_{i\in I})$ be a w.e.i. $\K$-algebra. Then, the functor
$$\delta:\mathrm{Mod}(\mathfrak{R}(\Lambda))\to \Mod(\Lambda), \quad M\mapsto \oplus_{i\in I}\,M(e_i),$$
is an isomorphism of categories and  $\delta(Y(e_i))=\Lambda e_{i},$ for any $i\in I.$
\end{proposition}

Let $(\Lambda,\{e_i\}_{i\in I})$ be a w.e.i $\K$-algebra and  $\Theta$ be a class of $R$-modules in $\Mod\,(R).$  We denote by 
$\F_f(\Theta)$ the full subcategory of $\Mod\,(R)$ whose objects have a finite $\Theta$-filtration. That is, $M\in \F_f(\Theta)$ if there exists a finite chain 
$$0=M_0\subseteq M_1\subseteq \dots\subseteq M_{t-1}\subseteq M_t=M$$
of submodules of $M$ such that each factor $M_i/M_{i-1}$ is isomorphic to some object in $\Theta.$  The modules in $\F_f(\Theta)$ are called 
$\Theta$-{\it{good modules}} and the class $\F_f(\Theta)$ is known as the  category of  $\Theta$-good modules.
\

Given a set $\X\subseteq \Mod\,(R)$ and $M\in \Mod\,(R),$ we recall that the trace of $\X$ in $M$ is the submodule 
$\Tr_{\mathcal{X}}(M):=\!\!\!\!\!\!
\!\!\!\!\!\!\displaystyle\sum _{\{f\in\mathrm{Hom}(X, M)\;\mid\; X\in \mathcal{X}\}}
\!\!\!\!\!\!\!\!\!\!\!\!\!\!\!\!\!\!\mathrm{Im}(f).$

In what follows, we define the notions of standard module and standardly stratified $\K$-algebra. Since we are interested only in basic 
w.e.i. $\K$-algebras, by taking into account the general definition of 
\cite{MOSS},  these notions are given below only for this class of algebras .

Let $(R,\{e_i\}_{i\in I})$ be a basic w.e.i $\K$-algebra.  Then, by \cite[Corollary 6.6(c)]{MOSS}, the set of finitely generated indecomposable projective $R$-modules (up to isomorphisms)  $\ind\,\proj(R)$ is given by the set
$\{Re_i\}_{i\in I}.$ For an ordinal number $\alpha,$ choose a partition (of size $\alpha$) $\tilde{\A}=\{\tilde{\A}_i\}_{i<\alpha}$ of the set $I.$  Define, for any $t\in\tilde{\A}_i,$ the $R$-module $P_t(i)={}_{\tilde{\A}}P_t(i):=Re_t$ which is known as the $t$-th indecomposable projective laying in the $i$-th level of the partition $\tilde{\A}.$   Let $P={}_{\tilde{\A}}P:=\{P(i)\}_{i<\alpha},$ where $P(i)={}_{\tilde{\A}}P(i):=\{P_t(i)\}_{t\in\tilde{\A}_i}.$ We say that ${}_{\tilde{\A}}P$ is the $\tilde{\A}$-arrangement, associated with the partition $\tilde{\A},$ of the set  $\ind\,\proj(R).$  We define the 
 family of $\tilde{\A}$-standard left $R$-modules $\Delta={}_{\tilde{\A}}\Delta:=\{\Delta(i)\}_{i<\alpha},$ where
 $\Delta(i)={}_{\tilde{\A}}\Delta(i):=\{\Delta_{t}(i)\}_{t\in \tilde{\A}_i}$ is defined as follows
 $$\Delta_{t}(i)={}_{\tilde{\A}}\Delta_{t}(i):=\frac{P_{t}(i)}{\mathrm{Tr}_{\oplus_{j<i}\overline{P}(j)}(P_t(i))},$$
 where $\overline{P}(j):=\bigoplus_{r\in\tilde{\A}_j}\,P_r(j).$

\begin{Remark}\label{Rcue} \cite[Remark 2.5]{MOSS} With the notation above, we have that the class $\F_f({}_{\tilde{\A}}\Delta)$ is closed under extensions in $\Mod\,(R).$
\end{Remark}

\begin{definition}\cite[Definition 6.11]{MOSS}
Let $(R,\{e_i\}_{i\in I})$ be a basic w.e.i $\K$-algebra. We say that the pair $(R,\tilde{\A})$ is a left {\bf standardly  stratified $\K$-algebra} if  $\tilde{\A}=\{\tilde{\A}_i\}_{i<\alpha}$ is a partition of the set $I$  such that
\begin{center}
$\Tr_{\oplus_{j<i}\overline{P}(j)}(P_{t}(i))\in \F_f(\bigcup_{j<i}\Delta(j)),$
for any $i<\alpha$ and $t\in\tilde{\A}_i.$
\end{center}
\end{definition}

\begin{proposition}\cite{MOSS} \label{BPSSR} For a basic w.e.i $\K$-algebra $(R,\{e_i\}_{i\in I})$ and a partition $\tilde{\A}=\{\tilde{\A}_i\}_{i<\alpha}$ of the set $I$ such  that $(R,\tilde{\A})$ is a left standardly  stratified $\K$-algebra, the following statements hold true.
\begin{itemize}
\item[(a)]  $\End_R({}_{\tilde{\A}}\Delta_t(i)))$ is a local ring, for each $i<\alpha$ and $t\in\tilde{\A}_i.$
\item[(b)]  For any $M\in\F_f({}_{\tilde{\A}}\Delta),$ the filtration multiplicity $[M:{}_{\tilde{\A}}\Delta_t(i))]$ does not depend on a given  ${}_{\tilde{\A}}\Delta$-filtration of $M.$
\item[(c)] $\F_f({}_{\tilde{\A}}\Delta)\subseteq\finp(R)$ and it is a locally finite $\K$-ringoid.
\item[(d)] $\F_f({}_{\tilde{\A}}\Delta)$ is closed under kernels of epimorphisms in $\Mod\,(R).$
\end{itemize}
\end{proposition}
\begin{proof} For the proof of  (a), (b) and (c), see \cite[Corollary 6.12]{MOSS}. Finally, for the proof of (d) use \cite[Remark 6.11, Proposition 4.12]{MOSS}.
\end{proof}

The following is a very useful criterion in order to see that a  basic w.e.i $\K$-algebra is standardly stratified.

\begin{proposition} \label{CriterioSS} For a basic w.e.i $\K$-algebra $(R,\{e_i\}_{i\in I})$ and a partition $\tilde{\A}=\{\tilde{\A}_i\}_{i<\alpha}$ of the set $I,$ the following statements are equivalent.
\begin{itemize}
\item[(a)] $(R,\tilde{\A})$ is a left standardly  stratified $\K$-algebra.
\item[(b)] ${}_{\tilde{\A}}P_t(i)\in\F_f({}_{\tilde{\A}}\Delta)$ for all $i<\alpha$ and $t\in\tilde{\A}_i.$
\end{itemize}
\end{proposition}
\begin{proof} (a) $\Rightarrow$ (b): It follows from Remark \ref{Rcue} and the definition of the $\tilde{\A}$-standard left $R$-modules.
\

 (b) $\Rightarrow$ (a): Let $i<\alpha,$  $t\in\tilde{\A}_i$ and $M:={}_{\tilde{\A}}P_t(i)\in\F({}_{\tilde{\A}}\Delta).$ By following 
 \cite[Definition 4.2]{MOSS}, for each $l<\alpha,$ we have the additive pre-radical functors
 $$\tau_l(-):=\Tr_{\oplus_{j\leq l}{}_{\tilde{\A}}\overline{P}(j)}(-)\quad\text{and}\quad \overline{\tau}_l(-):=\Tr_{\oplus_{j< l}{}_{\tilde{\A}}\overline{P}(j)}(-).$$
Consider the set $S:=\{l<\alpha \;:\; \overline{\tau}_l(M)\in\F_f(\bigcup_{k<l}{}_{\tilde{\A}}\,\Delta(k))\}$ and the quotient functor 
$Q_l:=\tau_l/  \overline{\tau}_l.$ By using transfinite induction, we will prove in what follows that $S=[0,\alpha)$ and thus  the result will be true. From \cite[Remark 6.11, Remark 4.8]{MOSS} and  \cite[Theorem 4.7]{MOSS}, we get the following:
\begin{itemize}
\item[(1)] For any $l<\alpha,$ the exact sequence $0\to  \overline{\tau}_l(M)\to \tau_l(M)\to Q_l(M)\to 0$ satisfies that 
$Q_l(M)\in{}_{\tilde{\A}}\Delta(l)^\oplus.$
\item[(2)] There is a finite number of ordinal numbers $0=k_0<k_1<k_2<\cdots<k_n<\alpha$ and a finite chain 
$0=M_0\subseteq M_1\subseteq M_2\subseteq\cdots\subseteq M_n=M$ of $R$-submodules of $M,$ with 
$M_t\in\F_f(\bigcup_{j\leq k_t}\,{}_{\tilde{\A}}\Delta(j))$ for all  $t,$ 
 satisfying the following properties
 \begin{itemize}
 \item[(i)] $\tau_j(M)=0$ for all $j\in[0,k_1);$
  \item[(ii)] $\tau_j(M)=M_t=\tau_{k_t}(M)$ for all $j\in[k_t,k_{t+1})$ and $t\in[1,n);$
  \item[(iii)] $\tau_j(M)=M_n=\tau_{k_n}(M)$ for all $j\geq k_n.$
 \end{itemize}
\end{itemize}
Now, we are ready to start with the procedure of transfinite induction. Note firstly that $0\in S$ since $\overline{\tau}_0(M)=0$ and 
$\F_f(\emptyset)=\{0\}.$ 
\

Let $l\in S$ be such that $l+1<\alpha.$ Since $l\in S,$ by the exact sequence of (1), it follows that 
$\overline{\tau}_{l+1}(M)=\tau_l(M)\in\F_f(\bigcup_{k<l+1}{}_{\tilde{\A}}\,\Delta(k))$
and thus $l+1\in S.$
\

Let $\gamma<\alpha$ be a limit ordinal and assume that $l\in S,$ for all $l<\gamma.$ We have to show that $\gamma\in S.$ Indeed, from \cite[Lemma 4.3(b)]{MOSS},  $\overline{\tau}_{\gamma}(M)=\sum_{l<\gamma}\,\tau_l(M).$ If $\gamma>k_n$ then  
$\overline{\tau}_{\gamma}(M)=\tau_{k_n}(M)\in\F_f(\bigcup_{j\leq k_n}{}_{\tilde{\A}}\,\Delta(j))$ and thus $\gamma\in S.$
\

Assume there exists some $t\in[1,n]$ such that $\gamma\leq k_t.$ If $\gamma=k_t,$ then 
$\overline{\tau}_{\gamma}(M)=M_{t-1}\in \F_f(\bigcup_{j\leq k_{t-1}}{}_{\tilde{\A}}\,\Delta(j))$ and thus $\gamma\in S$ since 
$k_{t-1}<k_t=\gamma.$
\

 Let $\gamma<k_t.$ If $t=1$ then $\overline{\tau}_{\gamma}(M)=0$ and there is nothing to prove. Then we may assume now that 
 $k_1\leq\gamma<k_t.$ If $\gamma=k_1$ then $\overline{\tau}_{\gamma}(M)=0$ and there is nothing to prove. Then, there is some 
 $2\leq m<n$ such that $k_m\leq\gamma<k_{m+1}.$ Hence $\overline{\tau}_{\gamma}(M)=M_{m-1}\in \F_f(\bigcup_{j\leq k_{m-1}}{}_{\tilde{\A}}\,\Delta(j))$ and thus $\gamma\in S$ since 
$k_{m-1}<k_m\leq\gamma.$
\end{proof}

\begin{Remark} Let $(R,\{e_i\}_{i\in I})$ be a Hom-finite w.e.i. $\K$-algebra,  $\tilde{\A}=\{\tilde{\A}_i\}_{i<\alpha}$ be a partition of the set $\ind\, \{e_i\}_{i\in I}$ (as in the general setting of \cite[Definition 6.10]{MOSS}) and let ${}_{\tilde{\A}}\Delta$ be the $\tilde{\A}$-standard left $R$-modules. Then, the proof of Proposition \ref{CriterioSS} allow us to conclude that the following statements are true.
\begin{itemize}
\item[(a)] $(R,\tilde{\A})$ is a left standardly  stratified $\K$-algebra if, and only if, ${}_{\tilde{\A}}P_t(i)\in\F_f({}_{\tilde{\A}}\Delta)$ for all $i<\alpha$ and $t\in\tilde{\A}_i.$
\item[(b)] $\Tr_{\oplus_{j< l}{}_{\tilde{\A}}\overline{P}(j)}(M)\in \F_f(\bigcup_{k<l}{}_{\tilde{\A}}\,\Delta(k))\},$ for all $l<\alpha$  and $M\in\F_f({}_{\tilde{\A}}\Delta).$
\end{itemize}

\end{Remark}
\section{Lower triangular matrix $\K$-algebras}
In this section, we introduce the notion of lower triangular $\K$-algebra and study conditions under which such an algebra is basic and standardly stratified in the sense of \cite{MOSS}. In order to do that, we need to work with disjoint unions, sum of ordinal numbers and partitions of a non-empty set.

{\sc Disjoint unions and partitions.} Let $I$ and $J$ be non-empty sets. The disjoint union of $I$ and $J$ is the set $I\vee J:=(I\times\{1\})\cup(J\times\{2\}).$ Then, we have two bijective functions $p_1:I\times\{1\}\to I,$ $(i,1)\mapsto i,$ and 
$p_2:J\times\{2\}\to J,$ $(j,2)\mapsto j.$ For two ordinal numbers $\alpha$ and $\beta,$ we consider the intervals of ordinal numbers $[0,\alpha)$ and $[0,\beta).$ In the set $[0,\alpha)\vee [0,\beta),$ we consider an order $\leq$ in which every element
of $[0,\alpha)\times\{1\}$ is smaller than any element in $[0,\beta)\times\{2\}.$ Thus, the ordinal type of the well-ordered set 
$([0,\alpha)\vee [0,\beta), \leq)$ is $\alpha+\beta.$
\

We denote by $\wp(I)$ the set of all the partitions $\tilde{\A}=\{\tilde{\A}_i\}_{i<\alpha}$ of the set $I,$ where $\alpha$ is an ordinal number. The {\bf joint partition function} 
$$\vee:\wp(I)\times \wp(J)\to \wp(I\vee J),\quad(\tilde{\A},\tilde{\B})\mapsto  \tilde{\A}\vee\tilde{\B}$$
is constructed as follows. For $\tilde{\A}=\{\tilde{\A}_i\}_{i<\alpha}\in\wp(I)$ and $\tilde{\B}=\{\tilde{\B}_j\}_{j<\beta}\in\wp(J),$ we set $ \tilde{\A}\vee\tilde{\B}=\{( \tilde{\A}\vee\tilde{\B})_t\}_{t\in[0,\alpha)\vee [0,\beta)},$ where
\begin{equation*}
( \tilde{\A}\vee\tilde{\B})_t:=\begin{cases} \tilde{\A}_i\times\{1\} & \text{if } t=(i,1), \\\tilde{\B}_i\times\{2\} & \text{if } t=(j,2). \end{cases}
\end{equation*}
We also have the {\bf splitting partition function} 
$$\spl:\wp(I\vee J)\to \wp(I)\times \wp(J),\quad \tilde{\C}\mapsto (\spl_1(\tilde{\C}), \spl_2(\tilde{\C})),$$
which is constructed, by using the given above bijections $p_1:I\times\{1\}\to I$ and $p_2:J\times\{2\}\to J,$  as follows:
$$\spl_1(\tilde{\C}):=p_1(\tilde{\C}\cap(I\times\{1\})\;\text{ and }\; \spl_2(\tilde{\C}):=p_2(\tilde{\C}\cap(J\times\{2\}).$$

\begin{lemma}\label{JSLema} For any non-empty sets $I$ and $J,$ the joint partition function $\vee:\wp(I)\times \wp(J)\to \wp(I\vee J)$ is a bijection whose inverse is the splitting partition function $\spl:\wp(I\vee J)\to \wp(I)\times \wp(J).$
\end{lemma}
\begin{proof} It is straightforward from the definitions involved.
\end{proof} 

{\sc Notation used for lower triangular matrices.}

Let  $(U,\{e_{i}\}_{i\in I})$ and 
$(T,\{f_{j}\}_{j\in J})$ be w.e.i. $\K$-algebras and $M$ be a $U$-$T$-bimodule acting centrally on $\K.$ We fix the following notation:
\
 
(1) the {\it{lower triangular matrix}} $\K$-algebra
$$\Lambda:=\left[\begin{smallmatrix}
T & 0 \\ 
M & U
\end{smallmatrix}\right]:=\{\left[\begin{smallmatrix} t & 0 \\  m & u \end{smallmatrix}\right]\;:\; t\in T, u\in U, m\in M\} $$
with sum and product  defined as follows
\begin{enumerate}
\item [(a)] $\left[\begin{smallmatrix}
t & 0 \\ 
m & u
\end{smallmatrix}\right]+\left[\begin{smallmatrix}
t' & 0\\ 
m' & u'
\end{smallmatrix}\right]=\left[\begin{smallmatrix}
t+t' & 0\\ 
m+m' & u+u'
\end{smallmatrix}\right],$

\item [(b)] $\left[\begin{smallmatrix}
t & 0 \\ 
m & u
\end{smallmatrix}\right]\left[\begin{smallmatrix}
t' & 0  \\ 
m' & u'
\end{smallmatrix}\right]=\left[\begin{smallmatrix}
tt' & 0\\
mt'+um' & uu'
\end{smallmatrix}\right];$
\end{enumerate}

(2) $\overline{e}_i:=\left[\begin{smallmatrix}
0 & 0 \\ 
0 & e_{i}
\end{smallmatrix}\right]\in \Lambda,$ for each $i\in I;$
\vspace{0.2cm}

(3) $\overline{f}_{j}=\left[\begin{smallmatrix}
f_{j}  & 0 \\ 
0 & 0
\end{smallmatrix}\right] \in \Lambda,$ for each $j\in J;$
\vspace{0.2cm}

(4) the family $\{g_r\}_{r\in I\vee J}$ in $\Lambda,$ where
\begin{equation*}
g_r=\begin{cases} \overline{e}_i & \text{if}\;  r=(i,1),\\ \overline{f}_j & \text{if}\;  r=(j,2). \end{cases}
\end{equation*}

By using the notation above, we can show the following lemma.

\begin{lemma}\label{LBasicoTM}  For the lower triangular matrix $\K$-algebra $\Lambda,$ the following equalities hold true.
\begin{itemize}
\item[(a)] $\overline{e}_i\Lambda = \left[\begin{smallmatrix} 0 & 0 \\ e_i M & e_i U \end{smallmatrix}\right]$ and  
$\overline{f}_j\Lambda = \left[\begin{smallmatrix} f_j T & 0 \\ 0 & 0 \end{smallmatrix}\right]$
$\forall\,i\in I$ and $\forall\,j\in J.$
\item[(b)] $\Lambda \overline{e}_i= \left[\begin{smallmatrix} 0 & 0 \\ 0 &  Ue_i \end{smallmatrix}\right]$ and  
$\Lambda \overline{f}_j = \left[\begin{smallmatrix} T f_j  & 0 \\ M f_j & 0 \end{smallmatrix}\right]$
$\forall\,i\in I$ and $\forall\,j\in J.$
\item[(c)] For $i,i'\in I$ and $j,j'\in J,$ we have that $\overline{f_j}\Lambda \overline{e}_i=0$ and 
$$\overline{e}_i\Lambda \overline{e}_{i'}=\left[\begin{smallmatrix} 0 & 0 \\ 0  & e_i Ue_{i'} \end{smallmatrix}\right],\; 
\overline{e}_i\Lambda \overline{f}_j=\left[\begin{smallmatrix} 0 & 0 \\ e_i M f_j  & 0 \end{smallmatrix}\right],\; 
\overline{f}_j\Lambda \overline{f}_{j'}=\left[\begin{smallmatrix} f_j T f_{j'} & 0 \\ 0  & 0 \end{smallmatrix}\right].
$$
\item[(d)] $\Hom_\Lambda(\Lambda \overline{f_j},\Lambda \overline{e}_i)=0$ $\forall\,(i,j)\in I\times J.$
\end{itemize}
\end{lemma}
\begin{proof} The proof of (a), (b) and (c) is a straightforward calculation by using matrix operations. Finally, (d) follows from (c) since 
$\Hom_\Lambda(\Lambda \overline{f_j},\Lambda \overline{e}_i)\simeq \overline{f_j}\Lambda \overline{e}_i$ as $\K$-modules.
\end{proof}

{\sc Basic properties of the triangular algebra $\mathbf{\Lambda}.$}

We start by proving the following fundamental lemma.

\begin{lemma}\label{LBasic1}
Let $(U,\{e_{i}\}_{i\in I})$ and $(T,\{f_{j}\}_{j\in J})$ be basic w.e.i. $\K$-algebras and $M$ be an $U$-$T$-bimodule such that 
$\{e_iMf_j\}_{(i,j)\in I\times J}\subseteq \fl(\K).$ Then, the lower  triangular matrix $\K$-algebra $(\Lambda, \{g_r\}_{r\in I\vee J})$ is basic and with enough idempotents.
\end{lemma}
\begin{proof}
It is clear that $\{g_r\}_{r\in I\vee J}$ is a family of orthogonal idempotents of $\Lambda.$ Then,   in order to show that $(\Lambda, \{g_r\}_{r\in I\vee J})$ is a w.e.i. $\K$-algebra, we need to prove that 
$\sum_{r\in I\vee J}\,g_r\Lambda= \Lambda=\sum_{r\in I\vee J}\,\Lambda g_r.$ Indeed, by Lemma \ref{LBasicoTM} (a) and 
Lemma \ref{primero}, we get the equalities
$$\sum_{r\in I\vee J}\,g_r\Lambda=\sum_{i\in I}\left[\begin{smallmatrix} 0 & 0 \\ e_i M   & e_i U\end{smallmatrix}\right]+
\sum_{j\in J}\left[\begin{smallmatrix} f_jT & 0 \\ 0   & 0\end{smallmatrix}\right]=\left[\begin{smallmatrix} T & 0 \\ M   & U\end{smallmatrix}\right]=\Lambda.$$
Analogously, we can show that $\Lambda=\sum_{r\in I\vee J}\,\Lambda g_r.$
\

Note that Lemma \ref{LBasicoTM} (c) implies us that $(\Lambda, \{g_r\}_{r\in I\vee J})$ is Hom-finite. Let us prove now that it is also basic. Indeed, by Lemma \ref{LBasicoTM} (c),  we get the ring isomorphisms $\overline{e}_i\Lambda \overline{e}_{i}\simeq e_i Ue_i$ and 
$\overline{f}_j\Lambda \overline{f}_{j}\simeq f_jTf_j$ $\forall\,i\in I,$ $\forall\, j\in L;$ and thus each idempotent $g_r$ is primitive.  To finish the proof, we need to show that $\{\Lambda g_r\}_{r\in I\vee J}$ are pairwise non isomorphic $\Lambda$-modules. In order to do that, we need to consider only the following two cases since, from Lemma \ref{LBasicoTM} (d), it is clear that $\Lambda \overline{e}_i\not\simeq\Lambda\overline{f}_j$ 
$\forall\, i\in I,$ $\forall\,j\in J.$
\

Case 1: Let $i,i'\in I$ be such that $i\neq i'.$ Suppose that $\Lambda \overline{e}_i\simeq \Lambda \overline{e}_{i'}$ as $\Lambda$-modules. Then, by Lemma \ref{LBasicoTM} (b), we know that $\Lambda \overline{e}_i=\left[\begin{smallmatrix} 0 & 0 \\ 0 &  Ue_i \end{smallmatrix}\right],$ and thus there is an isomorphism $Ue_i\simeq U e_{i'}$ of $U$-modules, contradicting that $(U,\{e_{i}\}_{i\in I})$ is basic. 
\

Case 2: Let $j,j'\in J$ be such that $j\neq j'.$ Suppose that $\Lambda \overline{f}_j\simeq \Lambda \overline{f}_{j'}$ as $\Lambda$-modules. Then, by \cite[Lemma 6.3]{MOSS}, we get that $\overline{f}_j\Lambda \simeq \overline{f}_{j'}\Lambda $ as $\Lambda$-modules. On the other hand, by Lemma \ref{LBasicoTM} (a), we know that 
$\overline{f}_j\Lambda =\left[\begin{smallmatrix} f_jT & 0 \\ 0 &  0 \end{smallmatrix}\right]$ and thus there is an isomorphism $f_j T\simeq f_{j'} T $ of $T$-modules, contradicting that $(T,\{e_{i}\}_{i\in I})$ is basic (see Remark \ref{Rbasic} (1)).
\end{proof}

Let $\A$  be an abelian category and $\X,$ $\Y$ be classes of objects in $\A.$ We consider the class $\X\star\Y$ whose objects are all the objects $M\in\A$ appearing in the middle term of an exact sequence $0\to X\to M\to Y\to 0,$ for some $X\in\X$ and $Y\in\Y.$ We recall the following result due to Ringel \cite[Theorem]{Ringel1} in the realm of finitely generated modules over Artin algebras. For the convenience of the reader we include a proof.

\begin{lemma}\label{lemRingel} Let $\A$ be an abelian category and $\X,$ $\Y$ be classes of objects in $\A$ which are closed under extensions and $\Ext^1_\A(\X,\Y)=0.$ Then, the class $\X\ast\Y$ is closed under extensions. 
\end{lemma}
\begin{proof}
Let  $0\to L\xrightarrow{f} M\xrightarrow{g} N\to 0$ be an exact sequence in $\A, $ with $L,N\in \X\ast\Y.$ Then, we have exact sequences  $0\to L_1\xrightarrow{\alpha_1} L\to L_2\to 0$ and $0\to N_1\xrightarrow{\beta_1} N\to N_2\to 0,$
with $L_{1},N_{1}\in \X$ and $L_{2},N_{2}\in \Y$. Thus, we have the pullback-pushout diagram
$$\xymatrix{& & 0\ar[d] & 0\ar[d]\\
0\ar[r] & L\ar@{=}[d]\ar[r]^{f'}  & M'\ar[r]^{g'}\ar[d]^{h} & N_{1}\ar[d]^{\beta_{1}}\ar[r]  &  0\\
0\ar[r] & L\ar[r]^{f} & M\ar[d]\ar[r]^{g}& N\ar[d]\ar[r] & 0\\
 & & N_{2}\ar@{=}[r]\ar[d] & N_{2}\ar[d]\\
  & & 0 & 0 }$$ and  the following commutative diagram

$$\xymatrix{0\ar[r] & L_{1}\ar[d]^{\alpha_{1}}\ar[r]^{f'\alpha_{1}}  & M'\ar[r]^{\gamma}\ar@{=}[d] & N_{1}'\ar[d]^{\beta_{1}'}\ar[r]  &  0\\
0\ar[r] & L\ar[r]^{f'} & M'\ar[r]^{g'} & N_{1}\ar[r] & 0.}$$ 
By snake lemma, $\mathrm{Ker}(\beta_{1}')=\mathrm{Coker}(\alpha_{1})=L_{2}$.
Then we have exact sequence
$$\eta:\quad \xymatrix{0\ar[r] & L_{2}\ar[r] & N_{1}'\ar[r]^{\beta_{1}'} & N_{1}\ar[r] & 0}.$$
Note that $\eta$ splits since $\mathrm{Ext}^{1}_\A(\X,\Y)=0.$ Thus, we get the exact sequence
$$\xymatrix{0\ar[r] & N_{1}\ar[r]^{\theta} & N_{1}'\ar[r]^{\psi} & L_{2}\ar[r] & 0}$$
and hence we get the following  pullback-pushout diagram $$\xymatrix{& & 0\ar[d] & 0\ar[d]\\
0\ar[r] & L_{1}\ar@{=}[d]\ar[r]  & M''\ar[r]\ar[d]^{\theta'} & N_{1}\ar[d]^{\theta}\ar[r]  &  0\\
0\ar[r] & L_{1}\ar[r]^{f'\alpha_{1}} & M'\ar[r]^{\gamma}\ar[d] & N_{1}'\ar[r]\ar[d] & 0\\
& & L_{2}\ar[d]\ar@{=}[r] & L_{2}\ar[d]\\
& & 0 & 0}$$ 
Since $L_{1},N_{1}\in \X,$ we have that $M''\in \X$. Now, we consider the commutative diagram
$$\xymatrix{0\ar[r] & M''\ar[r]^{\theta'}\ar@{=}[d] & M'\ar[d]^{h}\ar[r] & L_{2}\ar[r]\ar[d] & 0\\
0\ar[r] & M''\ar[r]_{h\theta'} & M\ar[r] & Z\ar[r] & 0}$$
which can be completed to the following commutative one

$$\xymatrix{& & 0\ar[d] & 0\ar[d]\\
0\ar[r] & M''\ar[r]^{\theta'}\ar@{=}[d] & M'\ar[d]^{h}\ar[r] & L_{2}\ar[r]\ar[d] & 0\\
0\ar[r] & M''\ar[r]_{h\theta'} & M\ar[r]\ar[d] & Z\ar[r]\ar[d] & 0\\
& & N_{2}\ar@{=}[r]\ar[d] & N_{2}\ar[d]\\
& & 0 & 0}$$
Since $L_{2},N_{2}\in \Y,$ we have that $Z\in \Y.$ Therefore, the exact sequence
$$\xymatrix{0\ar[r]  & M''\ar[r]^{h\theta'} & M\ar[r] & Z\ar[r] & 0}$$ give us that
$M\in \X\ast\Y,$ proving the result.
\end{proof}

{\sc Description of the modules in lower triangular matrix rings}
In order to give a description of the modules over a lower triangular matrix $\K$-algebra, we need some results and notions from  \cite{LeOS}. So in what follows, we collect them.
\

Let  $\mathcal{U}$ and $\mathcal{T}$ be ringoids and the bimodule $M\in \Mod\,(\mathcal{U}\otimes \mathcal{T}^{op}).$ The \textbf{triangular matrix ringoid}
$\mathbf{\Lambda}=\left[ \begin{smallmatrix}
\mathcal{T} & 0 \\ M & \mathcal{U}
\end{smallmatrix}\right]$  is defined as follows \cite[Definition 3.4]{LeOS}. The objects of $\mathbf{\Lambda}$ are  matrices of the form $\left[
\begin{smallmatrix}
T & 0 \\ M & U
\end{smallmatrix}\right],$ 
with $ T\in \mathcal{T} $ and $ U\in \mathcal{U}.$  Given a pair of objects in
$\left[ \begin{smallmatrix}
T & 0 \\
M & U
\end{smallmatrix} \right] ,  \left[ \begin{smallmatrix}
T' & 0 \\
M & U'
\end{smallmatrix} \right]$ in
$\mathbf{\Lambda},$ the set of morphisms is

$$\mathrm{Hom}_{\mathbf{\Lambda}}\left (\left[ \begin{smallmatrix}
T & 0 \\
M & U
\end{smallmatrix} \right] ,  \left[ \begin{smallmatrix}
T' & 0 \\
M & U'
\end{smallmatrix} \right]  \right)  := \left[ \begin{smallmatrix}
\mathrm{Hom}_{\mathcal{T}}(T,T') & 0 \\
M(U',T) & \mathrm{Hom}_{\mathcal{U}}(U,U')
\end{smallmatrix} \right].$$
For a more detailed description of this matrix ringoid, we recommend the reader to see \cite{LeOS}. The bimodule $M\in \Mod\,(\mathcal{U}\otimes \mathcal{T}^{op})$ induces  a 
functor $E:\T^{op}\to \Mod\,(\mathcal{U}),\quad T\mapsto M(-,T).$ Moreover, by considering the Yoneda functor 
$Y:\T^{op}\to \Mod\,(\T),\quad T\mapsto\Hom_\T(T,-),$ there is a unique functor 
$\mathbb{F}:\Mod\,(\T)\to \Mod\,(\mathcal{U})$ with commutes with direct limits and such that $\mathbb{F}\circ Y=E$ \cite[section 5]{LeOS}. This functor $\mathbb{F}$ is called the tensor product and it is usually denoted by $M\otimes_{\T}-.$ 

\begin{theorem}\cite[Theorem 3.14,  Proposition 5.3 ]{LeOS}\label{catcoma}
Let $\mathcal{U}$ and $\mathcal{T}$ be ringoids and  $M\in \Mod\,(\mathcal{U}\otimes \mathcal{T}^{op}).$ Then, 
there exists an equivalence of categories
$$\mathrm{Mod}\Big(\left[ \begin{smallmatrix}
\mathcal{T} & 0 \\
M & \mathcal{U}
\end{smallmatrix} \right] \Big)\simeq \Big(\mathbb{F}(\mathrm{Mod}(\mathcal{T})), \mathrm{Mod}(\mathcal{U})\Big).$$
\end{theorem}
We recall that $ \Big(\mathbb{F}(\mathrm{Mod}(\mathcal{T})),\mathrm{Mod}(\mathcal{U})\Big)$ is the so called comma category. The objects of this category are triples $ (A,g,B)$ with $A\in \mathrm{Mod}(\mathcal{T}), B\in \mathrm{Mod}(\mathcal{U})$
and $ g:M\otimes_{\T}A\to B $ a morphism of $ \mathcal{U} $-modules. A morphism between two objects $ (A,g,B) $ and $ (A',g',B') $ is a pair of morphism $(\alpha,\beta),$ where $\alpha:A\to A'$ is a morphism of $\mathcal{T}$-modules and
$\beta:B\to B'$ is a morphism of $\mathcal{U}$-modules and such that the following diagram commutes

\[
\begin{diagram}
\node{M\otimes_{\T}A} \arrow{e,t}{M\otimes_{\T}\,\alpha}\arrow{s,l}{g}\node{  M\otimes_{\T}A'}\arrow{s,r}{g'}\\
\node{B} \arrow{e,b}{\beta}\node{B'.}
\end{diagram}
\]

Let  $(U,\{e_{i}\}_{i\in I})$ and $(T,\{f_{j}\}_{j\in J})$  be basic w.e.i. $\K$-algebras and  $M$ be an $U$-$T$-bimodule such that 
$\{e_iMf_j\}_{(i,j)\in I\times J}\subseteq \fl(\K).$ By Lemma \ref{LBasic1}, we know that the lower triangular $\K$-algebra $\Lambda=\left[\begin{smallmatrix}
T  & 0 \\ 
M &  U
\end{smallmatrix}\right]$
is basic and with enough idempotents $\{g_r\}_{r\in I\vee J}.$ We can construct the ringoids $\mathfrak{R}(U)$, $\mathfrak{R}(T)$ and  the bifunctor $\overline{M}:\mathfrak{R}(U)\otimes \mathfrak{R}(T)^{op}\to \Ab,\quad (e_i,f_j)\mapsto e_iMf_j.$ Note that the ringoid 
$\mathfrak{R}(\Lambda)$ is isomorphic to $\left[\begin{smallmatrix}
\mathfrak{R}(T)  & 0 \\ 
\overline{M} &  \mathfrak{R}(U)
\end{smallmatrix}\right]$ and thus it can be identified one with the other. Then, by Proposition \ref{WeiEquiv1} and Theorem  \ref{catcoma}, we have the equivalences of categories 
$$\Mod\,(\Lambda)\simeq \mathrm{Mod}\Big(\left[ \begin{smallmatrix}
\mathfrak{R}(T) & 0 \\
\overline{M} & \mathfrak{R}(U)
\end{smallmatrix} \right] \Big)\simeq \Big(\mathbb{F}(\mathrm{Mod}\,(\mathfrak{R}(T)), \mathrm{Mod}\,(\mathfrak{R}(U))\Big).$$
It can be also shown that the following diagram commutes
$$\xymatrix{\mathrm{Mod}\,(\mathfrak{R}(T))\ar[r]^{\overline{M}\otimes_{\mathfrak{R}(T)}-}\ar[d] & \mathrm{Mod}(\mathfrak{R}(U))\ar[d]\\
\mathrm{Mod}\,(T)\ar[r]^{M \otimes_{T}-} & \mathrm{Mod}\,(U),}$$
where the vertical arrows are the isomorphisms of categories given by Proposition \ref{WeiEquiv1}.
Therefore, the category $\Mod\,(\Lambda)$ is equivalent to the comma category 
$(M\otimes_T(\Mod\,(T)),\Mod\,(U)),$ and thus,
each $\Lambda$-module $L$ can be seen as triple $(Y, \varphi, X),$ where $Y\in \Mod\,(T),$ $X\in \Mod\,(U)$ and 
$\varphi:M\otimes_{T} Y\to X$ is a morphism of  $U$-modules.
\

For the lower triangular matrix $\K$-algebra $\Lambda=\left[\begin{smallmatrix}
T  & 0 \\ 
M &  U
\end{smallmatrix}\right],$ we have the canonical injective morphisms of $\K$-algebras
$$i_{U}:U\to \left[\begin{smallmatrix}
T  & 0 \\ 
M &  U
\end{smallmatrix}\right],\; u\mapsto \left[\begin{smallmatrix}
0  & 0 \\ 
0 &  u
\end{smallmatrix}\right], \text{ and } i_{T}:T\to \left[\begin{smallmatrix}
T  & 0 \\ 
M &  U
\end{smallmatrix}\right],\;t\mapsto\left[\begin{smallmatrix}
t  & 0 \\ 
0 &  0
\end{smallmatrix}\right];$$
and the canonical surjective  morphisms of $\K$-algebras
$$p_{U}:\left[\begin{smallmatrix}
T  & 0 \\ 
M &  U
\end{smallmatrix}\right]\to U,\; \left[\begin{smallmatrix}
t  & 0 \\ 
m &  u
\end{smallmatrix}\right]\mapsto u, \text{ and } p_{T}:\left[\begin{smallmatrix}
T  & 0 \\ 
M &  U
\end{smallmatrix}\right]\to T,\;\left[\begin{smallmatrix}
t  & 0 \\ 
m &  u
\end{smallmatrix}\right]\mapsto t.$$ 
Then, by Proposition \ref{CRF}, we get the $\K$-linear exact functors,  
\begin{center} $(i_{U})_*:\Mod\,(\Lambda)\to \Mod\,(U),\quad$ 
$(i_{T})_*:\Mod\,(\Lambda)\to \Mod\,(T),$
\end{center}
\begin{center} $(p_{U})_*:\Mod\,(U)\to\Mod\,(\Lambda),\quad$ 
$(p_{T})_*:\Mod\,(T)\to\Mod\,(\Lambda).$
\end{center}

\begin{proposition}\label{MUmaps} For the lower triangular matrix $\K$-algebra $\Lambda=\left[\begin{smallmatrix}
T  & 0 \\ 
M &  U
\end{smallmatrix}\right],$ the following statements hold true.
\begin{itemize}
\item[(a)] Let $L\in\Mod\,(\Lambda).$ Then, its corresponding triple $(Y, \varphi, X)$ in the comma category $(M\otimes_T(\Mod\,(T)),\Mod\,(U))\simeq \Mod\,(\Lambda)$ is  the following: $Y=(i_{T})_*(L),$  $X=(i_{U})_*(L)$ and $\varphi$ is obtained (by using the universal property of the tensor product) from the $\mathbb{Z}$-bilinear and $T$-balanced map $M\times Y\to  X,\quad (m,y)\mapsto \left[\begin{smallmatrix}
0  &  0\\ 
m & 0
\end{smallmatrix}\right]y.$
\item[(b)] The sequence of $\Lambda$-modules
$$\xymatrix{0\ar[r] & (Y_{1},\varphi_{1},X_{1})\ar[r]^{(\alpha_{1},\alpha_{2})} & (Y,\varphi,X)\ar[r]^{(\beta_{1},\beta_{2})} & (Y_{2},\varphi_{2},X_{2})\ar[r] & 0}$$ is exact in $\Mod\,(\Lambda)$ if, and only if, the following two sequences 
$$\xymatrix{0\ar[r] & X_{1}\ar[r]^{\alpha_{1}} & X\ar[r]^{\beta_{1}} & X_{2}\ar[r] & 0,}$$
$$\xymatrix{0\ar[r] & Y_{1}\ar[r]^{\alpha_{2}} & Y\ar[r]^{\beta_{2}} & Y_{2}\ar[r] & 0}$$
are exact, respectively, in $\Mod\,(U)$ and $\Mod\,(T).$ 
\item[(c)] For the change of rings functors, described above, we have \\
$(i_T)_*\circ (p_T)_*=1_{\Mod\,(T)},$ $(i_U)_*\circ (p_U)_*=1_{\Mod\,(U)}$ and 
\begin{center}
$(i_U)_*\circ (p_T)_*=0=(i_T)_*\circ (p_U)_*.$
\end{center}
\item[(d)] $(p_T)_*(\Mod\,(T))=\{ (Y,0,0)\;:\;Y\in\Mod\,(T)\}.$ Moreover, if $\Y$ is closed under extensions in $\Mod(T),$ then $(p_T)_*(\Y)$ is closed under extensions in $\Mod\,(\Lambda).$
\item[(e)] $(p_U)_*(\Mod\,(U))=\{ (0,0, X)\;:\;X\in\Mod\,(U)\}.$ Moreover, if $\X$ is closed under extensions in $\Mod\,(U),$ then $(p_U)_*(\X)$ is closed under extensions in $\Mod\,(\Lambda).$
\item[(f)] $\Ext^1_\Lambda((p_U)_*(\Mod\,(U)),(p_T)_*(\Mod\,(T))=0.$
\item[(g)] Let $\Y$ be closed under extensions in $\Mod(T)$ and $\X$ be closed under extensions in $\Mod\,(U).$ Then, the class $(p_U)_*(\X)\star (p_T)_*(\Y)$ is closed under extensions in $\Mod\,(\Lambda).$
\end{itemize}
\end{proposition}
\begin{proof} (a) It follows from \cite[Lemma 3.12, Lemma 3.13]{LeOS} and Proposition \ref{WeiEquiv1}.
\

(b) It follows from \cite[Proposition 5.3 and Proposition 5.3]{LeOS}.
\

(c) Let $Y\in\Mod\,(T).$ Then, as an abelian group, we have 
$$(p_T)_*(Y)=\sum_{r\in I\vee J}\,p_T(g_r)Y=\sum_{j\in J}f_jY=Y$$ and as a $\Lambda$-module $\left[\begin{smallmatrix}
t  & 0 \\ 
m &  u
\end{smallmatrix}\right]\cdot y=ty$ $\forall\,y\in Y.$ Moreover, as abelian group
$$(i_T)_*((p_T)_*(Y))=\sum_{j\in J}\,\overline{f}_j\cdot Y=\sum_{j\in J}f_jY=Y.$$ Furthermore, the structure of $T$-module on $(i_T)_*((p_T)_*(Y))$ coincide with the one on $Y,$ and thus 
$(i_T)_*\circ (p_T)_*=1_{\Mod\,(T)}.$ On the other hand,
$$(i_U)_*((p_T)_*(Y))=\sum_{i\in I}\,\overline{e}_i\cdot Y=0;$$ proving that $(i_U)_*\circ (p_T)_*=0.$ In a similar way, we can check the other equalities of (c).
\

(d) From (a) and (c), we get the equality in (d).  Let $\Y$ be closed under extensions on $\Mod(T).$ Consider an exact sequence of $\Lambda$-modules
$$\xymatrix{0\ar[r] & (Y_{1},0,0)\ar[r]^{(\alpha_{1},\alpha_{2})} & (Y,f,X)\ar[r]^{(\beta_{1},\beta_{2})} & (Y_{2},0,0)\ar[r] & 0,}$$ 
where $Y_1,Y_2\in\Y.$ Then by  (b), and since $\Y$ is closed under extensions,  we 
conclude that $Y\in \Y$ and $X=0.$ Therefore $(Y,f,X)\in (p_T)_*(\Y).$
\

(e) It follows as in the proof of (d).
\

(f) Consider the exact sequence of $\Lambda$-modules
$$\eta:\quad \xymatrix{0\ar[r] & (Y,0,0)\ar[r]^{(\alpha_{1},\alpha_{2})} & (Y',\varphi',X')\ar[r]^{(\beta_{1},\beta_{2})} & (0,0,X)\ar[r] & 0}.$$
Then, by (b), we get the split exact sequences
$$\xymatrix{0\ar[r] & Y\ar[r]^{\alpha_{1}} & Y'\ar[r] & 0\ar[r] & 0}\,\,\,\text{in}\,\, \Mod\,(T)$$
$$\xymatrix{0\ar[r] & 0\ar[r]^{} & X'\ar[r]^{\beta_{2}} & X\ar[r] & 0}\,\,\,\text{in}\,\, \Mod\,(U).$$
Thus, by (b) and the above split exact sequences, we obtain that $\eta$ splits. Therefore, we get (f) from (d) and (e).
\

(g) It follows from (d), (e), (f) and Lemma \ref{lemRingel}. 
\end{proof}

\begin{Remark}\label{MUmaps2} For the lower triangular matrix $\K$-algebra $\Lambda=\left[\begin{smallmatrix}
T  & 0 \\ 
M &  U
\end{smallmatrix}\right],$ we have the following.
\begin{itemize}
\item[(1)] Let $(Y,\varphi, X)\in(M\otimes_T(\Mod\,(T)),\Mod\,(U))\simeq \Mod\,(\Lambda).$ We can see the $\Lambda$-module $(Y,\varphi, X)$ as a matrix representation. Indeed, consider the set of matrices
$$\left[\begin{smallmatrix}
Y & 0 \\ 
X & 0
\end{smallmatrix}\right]:=\{\left[\begin{smallmatrix}
y & 0 \\ 
x & 0
\end{smallmatrix}\right]\;:\;x\in X, y\in Y\},$$
with the structure of $\Lambda$-module given by 
$$\left[\begin{smallmatrix}t & 0 \\ m & u\end{smallmatrix}\right]\cdot  \left[\begin{smallmatrix}y & 0 \\ x & 0\end{smallmatrix}\right]:=\left[\begin{smallmatrix}ty & 0 \\ \varphi(m\otimes y) +ux& 0\end{smallmatrix}\right].$$
Note that $(i_T)_*(\left[\begin{smallmatrix} Y & 0 \\ X & 0\end{smallmatrix}\right] )\simeq Y$ as $T$-modules and 
$(i_U)_*(\left[\begin{smallmatrix} Y & 0 \\ X & 0\end{smallmatrix}\right]) \simeq X$ as $U$-modules
\item[(2)] Let $Y$ be a $T$-submodule of $T$ and $X$ be an $U$-submodule of $M.$ Then, the map $i_{Y,X}:M\otimes_T\,Y\to X,$ 
$m\otimes y\mapsto my,$ is a morphism of $U$-modules and thus the triple $(Y,i_{Y,X},X)$ is a $\Lambda$-module that can be seen as 
the matrix $\left[\begin{smallmatrix} Y & 0 \\ X & 0\end{smallmatrix}\right].$
\item[(3)] Let $X\in\Mod\,(U).$ Then, the set of matrices $\left[\begin{smallmatrix} 0 & 0 \\ 0 & X\end{smallmatrix}\right]$ becomes a $\Lambda$-module via the action 
$\left[\begin{smallmatrix}t & 0 \\ m & u\end{smallmatrix}\right]\cdot  \left[\begin{smallmatrix}0 & 0 \\ 0 & x\end{smallmatrix}\right]:=\left[\begin{smallmatrix}0 & 0 \\ 0& ux\end{smallmatrix}\right].$ On the other hand, by (1), we have the $\Lambda$-module 
$\left[\begin{smallmatrix} 0 & 0 \\ X & 0\end{smallmatrix}\right].$ Finally, an easy calculation shows that 
$$\mu:\left[\begin{smallmatrix} 0 & 0 \\ X & 0\end{smallmatrix}\right]\to \left[\begin{smallmatrix} 0 & 0 \\ 0 & X\end{smallmatrix}\right],\quad \left[\begin{smallmatrix} 0 & 0 \\ x & 0\end{smallmatrix}\right]\mapsto  \left[\begin{smallmatrix} 0 & 0 \\ 0 & x\end{smallmatrix}\right],$$ is an isomorphism of $\Lambda$-modules.
\item[(4)] By Lemma \ref{LBasicoTM}, for any $i\in I$ and $j\in J,$ we get \\
{\small $(i_T)_*(\Lambda \overline{f}_j)\simeq Tf_j,$ $(i_T)_*(\Lambda \overline{e}_i)=0,$ $(i_U)_*(\Lambda \overline{f}_j)\simeq Mf_j$ and 
$(i_U)_*(\Lambda \overline{e}_i)\simeq Ue_i.$}
\end{itemize}
\end{Remark}

\begin{lemma}\label{LBasic2}
Let $(U,\{e_{i}\}_{i\in I})$ and $(T,\{f_{j}\}_{j\in J})$ be basic w.e.i. $\K$-algebras, $M$ be an $U$-$T$-bimodule such that 
$\{e_iMf_j\}_{(i,j)\in I\times J}\subseteq \fl(\K),$ $\tilde{\A}=\{\tilde{\A}_i\}_{i<\alpha}\in\wp(I)$ and 
$\tilde{\B}=\{\tilde{\B}_j\}_{i<\beta}\in\wp(J).$ Then, for
$\tilde{\C}:=\tilde{\A}\vee\tilde{\B}\in\wp(I\vee J)$  (see \ref{JSLema}), the corresponding $\tilde{\C}$-standard $\Lambda$-modules computed in the basic  lower triangular matrix $\K$-algebra $(\Lambda, \{g_r\}_{r\in I\vee J})$ (see \ref{LBasic1} ) are related with those of $(U,\tilde{\A})$ and $(T,\tilde{\B})$ as follows.
\begin{enumerate}
\item [(a)] If $t\in [0,\alpha)\times\{1\},$ then 
$${}_{\tilde{\C}}\Delta_r(t)\simeq (p_U)_*({}_{\tilde{\A}}\Delta_{p_1(r)}(p_1(t)))=\left[\begin{matrix}
0 & 0 \\ 
0 & {}_{\tilde{\A}}\Delta_{p_1(r)}(p_1(t))
\end{matrix}\right].$$

\item [(b)] If $t\in [0,\beta)\times\{2\},$ then 
$${}_{\tilde{\C}}\Delta_r(t)\simeq (p_T)_*({}_{\tilde{\B}}\Delta_{p_2(r)}(p_2(t)))=\left[\begin{matrix}
{}_{\tilde{\B}}\Delta_{p_2(r)}(p_2(t)) & 0\\ 
0 & 0
\end{matrix}\right].$$
\item [(c)] {\small $\F({}_{\tilde{\C}}\Delta)=(p_U)_*(\F({}_{\tilde{\A}}\Delta))\star (p_T)_*(\F({}_{\tilde{\B}}\Delta))=
\F((p_U)_*({}_{\tilde{\A}}\Delta))\star \F( (p_T)_*({}_{\tilde{\B}}\Delta)).$}
\vspace{0.1cm}
\item [(d)] $\F((i_T)_*({}_{\tilde{\C}}\Delta))=\F({}_{\tilde{\B}}\Delta)$ and $\F((i_U)_*({}_{\tilde{\C}}\Delta))=\F({}_{\tilde{\A}}\Delta).$
\item [(e)] Let $(Y,\varphi,X)\in\Mod\,(\Lambda).$ Then
$$(Y,\varphi,X)\in\F({}_{\tilde{\C}}\Delta)\;\Leftrightarrow\; Y\in\F({}_{\tilde{\B}}\Delta)\;\text{ and }\;X\in\F({}_{\tilde{\A}}\Delta).$$
\end{enumerate}
\end{lemma}
\begin{proof}
We recall that $\tilde{\C}=\{\tilde{\C}_t\}_{t\in[0,\alpha)\vee [0,\beta),}$ were 
$\tilde{\C}_t=\tilde{\A}_i\times\{1\}$ if $t=(i,1),$ and $\tilde{\C}_t=\tilde{\B}_j\times\{2\}$ if $t=(j,2).$ We will make use of the equalities given in Lemma \ref{LBasicoTM}. Let $t\in[0,\alpha)\vee [0,\beta)$ and $r\in\tilde{\C}_t.$ We have to compute 
$${}_{\tilde{\C}}\Delta_{r}(t):=\frac{{}_{\tilde{\C}}P_r(t)}{\Tr_{\oplus_{t'<t}\,{}_{\tilde{\C}}\overline{P}(t')}({}_{\tilde{\C}}P_r(t))},$$
 where ${}_{\tilde{\C}}\overline{P}(t'):=\bigoplus_{k\in\tilde{\C}_{t'}}\,{}_{\tilde{\C}}P_k(t').$ We also use that 
 $$(*)\quad \Tr_{\oplus_{t'<t}\,{}_{\tilde{\C}}\overline{P}(t')}({}_{\tilde{\C}}P_r(t))=\sum_{t'<t}\sum_{k\in\tilde{\C}_{t'}}\Tr_{{}_{\tilde{\C}}P_k(t')}({}_{\tilde{\C}}P_r(t)).$$

(a) Let $t\in [0,\alpha)\times\{1\}$ and $r\in \tilde{\C}_t=\tilde{\A}_{p_1(t)}\times\{1\}.$ Then, we have
$${}_{\tilde{\C}}P_r(t)=\Lambda g_r=\Lambda \overline{e}_{p_1(r)}=\left[\begin{smallmatrix} 0 & 0 \\ 0 &  U e_{p_1(r)} \end{smallmatrix}\right]=\left[\begin{matrix} 0 & 0 \\ 0 &  {}_{\tilde{\A}}P_{p_1(r)}(p_1(t)) \end{matrix}\right].$$
Since $t'<t$ and $ t\in [0,\alpha)\times\{1\},$ we have $t'\in [0,\alpha)\times\{1\}$ and thus
\begin{align*}
\Tr_{{}_{\tilde{\C}}P_k(t')}({}_{\tilde{\C}}P_r(t)) & = \Tr_{\Lambda \overline{e}_{p_1(k)}}(\Lambda\overline{e}_{p_1(r)})\\
&= \Lambda\overline{e}_{p_1(k)} \Lambda\overline{e}_{p_1(r)}\\
&=\left[\begin{smallmatrix} 0 & 0 \\ 0 &  U e_{p_1(k)}U e_{p_1(r)} \end{smallmatrix}\right]\\
& = \left[\begin{matrix} 0 & 0 \\ 0 &  \Tr_{{}_{\tilde{\A}}P_{p_1(k)}(p_1(t'))} ({}_{\tilde{\A}}P_{p_1(r)}(p_1(t)))\end{matrix}\right].
\end{align*}
Therefore, from (*) we conclude now that 
$$\Tr_{\oplus_{t'<t}\,{}_{\tilde{\C}}\overline{P}(t')}({}_{\tilde{\C}}P_r(t))=\left[\begin{matrix} 0 & 0 \\ 0 &  
\Tr_{\oplus_{p_1(t')<p_1(t)}\,{}_{\tilde{\A}}\overline{P}(p_1(t'))}({}_{\tilde{\A}}P_{p_1(r)}(p_1(t)))\end{matrix}\right];$$ and thus,
we get (a).
\

(b)  Let $t\in [0,\beta)\times\{2\}$ and $r\in \tilde{\C}_t=\tilde{\B}_{p_2(t)}\times\{2\}.$ Then, we have
$${}_{\tilde{\C}}P_r(t)=\Lambda g_r=\Lambda \overline{f}_{p_2(r)}=\left[\begin{smallmatrix} Tf_{p_2(r)} & 0 \\ Mf_{p_2(r)} & 0 \end{smallmatrix}\right]=\left[\begin{matrix} {}_{\tilde{\B}}P_{p_2(r)}(p_2(t)) & 0 \\ Mf_{p_2(r)} &  0 \end{matrix}\right].$$
We assert that the following equality is true
$$\Tr_{\oplus_{t'<t}\,{}_{\tilde{\C}}\overline{P}(t')}({}_{\tilde{\C}}P_r(t))=
\left[\begin{matrix}\Tr_{\oplus_{p_2(t')<p_2(t)}\,{}_{\tilde{\B}}\overline{P}(p_2(t'))}({}_{\tilde{\B}}P_{p_2(r)}(p_2(t)) & 0 \\ Mf_{p_2(r)} &  0 \end{matrix}\right],$$
and thus (b) will follows from this assertion. Let us prove the assertion. Indeed, note firstly that $m<t,$ for any $m\in[0,\alpha)\times\{1\},$ since $t\in[0,\beta)\times\{2\}.$ For $t'<t,$ we set $t':=(t'_1,t'_2)\in [0,\alpha)\vee[0,\beta).$ 
Hence  $\Tr_{\oplus_{t'<t}\,{}_{\tilde{\C}}\overline{P}(t')}({}_{\tilde{\C}}P_r(t))$ splits in two summands (see $(*)$)
$$(**)\quad  \sum_{t'<t, t'_2=1}\sum_{k\in\tilde{\C}_{t'}}\Tr_{{}_{\tilde{\C}}P_k(t')}({}_{\tilde{\C}}P_r(t))+\sum_{t'<t, t'_2=2}\sum_{k\in\tilde{\C}_{t'}}\Tr_{{}_{\tilde{\C}}P_k(t')}({}_{\tilde{\C}}P_r(t)).$$
Let us compute such summands, we start with the first one of $(**)$
\begin{align*}
I:=\sum_{t'<t, t'_2=1}\sum_{k\in\tilde{\C}_{t'}}\Tr_{{}_{\tilde{\C}}P_k(t')}({}_{\tilde{\C}}P_r(t)) & = 
\sum_{t'<t, t'_2=1}\sum_{k\in\tilde{\C}_{t'}}\Tr_{\Lambda{g_k}}(\Lambda{g_r})\\
&= \sum_{t'<t, t'_2=1}\sum_{k\in\tilde{\C}_{t'}} \left[\begin{smallmatrix} 0 & 0 \\ 0 &  U e_{p_1(k)} \end{smallmatrix}\right] \left[\begin{smallmatrix} T f_{p_2(r)} & 0 \\ Mf_{p_2(r)} &  0 \end{smallmatrix}\right]\\
&= \left[\begin{smallmatrix} 0 & 0 \\ Mf_{p_2(r)} &  0 \end{smallmatrix}\right].
\end{align*}
For the latest equality above, we used that 
$$\sum_{t'<t, t'_2=1}\sum_{k\in\tilde{\C}_{t'}} U e_{p_1(k)} Mf_{p_2(r)}=UMf_{p_2(r)}=Mf_{p_2(r)}.$$
We compute now the second summand of $(**)$ 
\begin{align*}
II:=\sum_{t'<t, t'_2=2}\sum_{k\in\tilde{\C}_{t'}}\Tr_{{}_{\tilde{\C}}P_k(t')}({}_{\tilde{\C}}P_r(t)) & =
\sum_{t'<t, t'_2=2}\sum_{k\in\tilde{\C}_{t'}}\Tr_{\Lambda{g_k}}(\Lambda{g_r})\\
& = \sum_{t'<t, t'_2=2}\sum_{k\in\tilde{\C}_{t'}} \left[\begin{smallmatrix} T f_{p_2(k)} & 0 \\ Mf_{p_2(k)} &  0 \end{smallmatrix}\right]\left[\begin{smallmatrix} T f_{p_2(r)} & 0 \\ Mf_{p_2(r)} &  0 \end{smallmatrix}\right]\\
& = \sum_{t'<t, t'_2=2}\sum_{k\in\tilde{\C}_{t'}}\left[\begin{smallmatrix} T f_{p_2(k)}T f_{p_2(r)} & 0 \\ Mf_{p_2(k)}T f_{p_2(r)} &  0 \end{smallmatrix}\right]
\end{align*}
Since $Mf_{p_2(k)}T f_{p_2(r)}\subseteq M f_{p_2(r)},$ by doing the sum $I+II,$ we get (b).
\

(c) By Remark \ref{Rcue}, we know that the classes $\F({}_{\tilde{\C}}\Delta),$ $\F({}_{\tilde{\A}}\Delta)$ and $\F({}_{\tilde{\B}}\Delta)$ are closed under extensions. Then, by Proposition \ref{MUmaps} (g), we get that the class $\Z:=(p_U)_*(\F({}_{\tilde{\A}}\Delta))\star (p_T)_*(\F({}_{\tilde{\B}}\Delta))$ is closed under extensions.
\

From (a) and (b) we have that ${}_{\tilde{\C}}\Delta\subseteq \Z$ and thus $\F({}_{\tilde{\C}}\Delta)\subseteq\Z$ since $\Z$ is closed under extensions. Now, we set $\mathcal{W}:=\F((p_U)_*({}_{\tilde{\A}}\Delta))\star \F( (p_T)_*({}_{\tilde{\B}}\Delta)).$ Since 
$(p_U)_*$ and $(p_T)_*$ are exact functors, by (a) and (b),  we get that 
\begin{equation*}
\begin{split}
(p_U)_*(\F({}_{\tilde{\A}}\Delta)) & \subseteq \F((p_U)_*({}_{\tilde{\A}}\Delta))\subseteq \F({}_{\tilde{\C}}\Delta),\\
 (p_T)_*(\F({}_{\tilde{\B}}\Delta)) & \subseteq \F((p_T)_*({}_{\tilde{\B}}\Delta))\subseteq \F({}_{\tilde{\C}}\Delta).
\end{split}
\end{equation*}
Therefore $\F({}_{\tilde{\C}}\Delta)\subseteq\Z\subseteq\mathcal{W}\subseteq \F({}_{\tilde{\C}}\Delta),$ proving (c).
\

(d) It follows from (a), (b) and Proposition \ref{MUmaps} (c).
\

(e) Let $Y\in\F({}_{\tilde{\B}}\Delta)$ and $X\in\F({}_{\tilde{\A}}\Delta).$ Thus, by (c),  $(Y,\varphi,X)\in\F({}_{\tilde{\C}}\Delta).$
\

Let $(Y,\varphi,X)\in\F({}_{\tilde{\C}}\Delta).$ Then, by (c), there is an exact sequence $0\to (0,0,X')\to (Y,\varphi,X)\to (Y',0,0)\to 0,$ where $X'\in \F({}_{\tilde{\A}}\Delta)$ and $Y'\in\F({}_{\tilde{\B}}\Delta).$ By using the exact sequence $0\to (0,0,X)\to (Y,\varphi,X)\to (Y,0,0)\to 0,$ we get from Proposition \ref{MUmaps} (b) that $X\simeq X'$ and $Y\simeq Y'.$ Therefore $Y\in\F({}_{\tilde{\B}}\Delta)$ and $X\in\F({}_{\tilde{\A}}\Delta).$
\end{proof}

\begin{theorem}\label{TP1} Let $(U,\{e_{i}\}_{i\in I})$ and $(T,\{f_{j}\}_{j\in J})$ be basic w.e.i. $\K$-algebras, $M$ be an $U$-$T$-bimodule 
such that 
$\{e_iMf_j\}_{(i,j)\in I\times J}\subseteq \fl(\K),$ $\tilde{\A}=\{\tilde{\A}_i\}_{i<\alpha}\in\wp(I)$ and 
$\tilde{\B}=\{\tilde{\B}_j\}_{i<\beta}\in\wp(J).$ Then, for
$\tilde{\C}:=\tilde{\A}\vee\tilde{\B}\in\wp(I\vee J)$  (see \ref{JSLema}), and the lower triangular matrix $K$-algebra 
$\Lambda:=\left[\begin{smallmatrix} T  & 0 \\ M &  U\end{smallmatrix}\right],$
the following statements are equivalent. 
\begin{itemize}
\item[(a)] $(\Lambda, \tilde{\C})$ is a left standardly stratified $\K$-algebra.
\item[(b)] $\{Mf_j\}_{j\in J}\subseteq\F({}_{\tilde{\A}}\Delta)$ and the pairs $(T, \tilde{\B})$ and $(U, \tilde{\A})$ are left standardly stratified $\K$-algebras.
\end{itemize}
\end{theorem}
\begin{proof} By Lemma \ref{LBasic1},  $(\Lambda, \{g_r\}_{r\in I\vee J})$ is a basic w.e.i. $\K$ algebra, where 
$g_r=\overline{e}_i$ if $r=(i,1),$ and $g_r=\overline{f}_j$ if $r=(j,2).$ Furthermore:
\

\begin{enumerate}
\item[(1)]  For $j\in J,$ we have 
 $\Lambda \overline{f}_j=\left[\begin{smallmatrix} T f_j & 0 \\ Mf_j &  0 \end{smallmatrix}\right]\simeq (Tf_j,\varphi, Mf_j)\in \Mod\,(\Lambda),$ where $\varphi:=i_{Tf_j,Mf_j},$ see Remark \ref{MUmaps2} (2);
 \item[(2)] For $i\in I,$ we have  $\Lambda \overline{e}_i=\left[\begin{smallmatrix} 0 & 0 \\ 0 &  Ue_i\end{smallmatrix}\right]\simeq (0,0, Ue_i)\in \Mod\,(\Lambda),$ see Remark \ref{MUmaps2}.
\end{enumerate}

 (a) $\Rightarrow$ (b): Let $(\Lambda, \tilde{\C})$ be a left standardly stratified $\K$-algebra. Then, by Proposition \ref{CriterioSS}, we have that $\{\Lambda g_r\}_{r\in I\vee J}\subseteq \F({}_{\tilde{\C}}\Delta).$ 
 \
 
 Let $j\in J.$  Since $(i_T)_*:\Mod\,(\Lambda)\to \Mod\,(T)$ is an exact functor, we get from (1), Remark \ref{MUmaps2} (1) and Lemma \ref{LBasic2} (d)
 $$Tf_j\simeq (i_T)_*(\Lambda \overline{f}_j)\in (i_T)_*(\F({}_{\tilde{\C}}\Delta))\subseteq \F((i_T)_*({}_{\tilde{\C}}\Delta))=\F({}_{\tilde{\B}}\Delta).$$
 Then, from Proposition \ref{CriterioSS}, it follows that $(T, \tilde{\B})$ is a left standardly stratified $\K$-algebra. On the other hand, using that $(i_U)_*:\Mod\,(\Lambda)\to \Mod\,(U)$ is an exact functor, 
Remark \ref{MUmaps2} (1) and Lemma \ref{LBasic2} (d), we get
$$Mf_j\simeq (i_U)_*(\Lambda \overline{f}_j)\in (i_U)_*(\F({}_{\tilde{\C}}\Delta))\subseteq \F((i_U)_*({}_{\tilde{\C}}\Delta))=\F({}_{\tilde{\A}}\Delta).$$
and thus $\{Mf_j\}_{j\in J}\subseteq\F({}_{\tilde{\A}}\Delta).$
\

Let $i\in I.$ Then, by the exactness of the functor $(i_U)_*:\Mod\,(\Lambda)\to \Mod\,(U),$ (2),
Remark \ref{MUmaps2} (1) and Lemma \ref{LBasic2} (d), 
$$Ue_i\simeq (i_U)_*(\Lambda \overline{e}_i)\in (i_U)_*(\F({}_{\tilde{\C}}\Delta))\subseteq \F((i_U)_*({}_{\tilde{\C}}\Delta))=\F({}_{\tilde{\A}}\Delta).$$
and hence, from Proposition \ref{CriterioSS}, we get that $(U, \tilde{\A})$ is a left standardly stratified $\K$-algebra.
\

(b) $\Rightarrow$ (a): Assume the hypothesis given in (b). In particular, by  Proposition \ref{CriterioSS}, we get that 
$\{Ue_i\}_{i\in I}\subseteq\F({}_{\tilde{\A}}\Delta)$ and $\{Tf_j\}_{j\in J}\subseteq\F({}_{\tilde{\B}}\Delta).$ In order to show that 
$(\Lambda, \tilde{\C})$ is a left standardly stratified $\K$-algebra, it is enough to prove that $\{\Lambda g_r\}_{r\in I\vee J}\subseteq \F({}_{\tilde{\C}}\Delta).$ 
\

Let $j\in J.$ Then, by (1) and Proposition \ref{MUmaps} (b),  there is an exact sequence of $\Lambda$-modules
$$\eta_j\;:\; 0\to (0,0,Mf_j)\to (Tf_j,\varphi, Mf_j)\to (Tf_j,0,0)\to 0.$$
On the other hand, by Proposition \ref{MUmaps} (e), 
$$(0,0,Mf_j)=(p_U)_*(Mf_j)\in(p_U)_*(\F({}_{\tilde{\A}}\Delta));$$ and by  Proposition \ref{MUmaps} (d), 
$$(Tf_j,0,0)=(p_T)_*(Tf_j)\in(p_T)_*(\F({}_{\tilde{\B}}\Delta)).$$ 
Thus, from (1),  Lemma \ref{LBasic2} (c) and $\eta_j,$ we get 
$$\Lambda \overline{f}_j\simeq (Tf_j,\varphi, Mf_j)\in  (p_U)_*(\F({}_{\tilde{\A}}\Delta))\star (p_T)_*(\F({}_{\tilde{\B}}\Delta))=
\F({}_{\tilde{\C}}\Delta).$$

 Let $i\in I.$  Then, by (2),  Proposition \ref{MUmaps} (e) and Lemma \ref{LBasic2} (c), we get 
 $$\Lambda\overline{e}_i\simeq (p_U)_*(Ue_i)\in(p_U)_*(\F({}_{\tilde{\A}}\Delta))\subseteq \F({}_{\tilde{\C}}\Delta);$$ and thus 
 the result follows.
\end{proof}

\section{Locally bounded $\K$-algebras with enough idempotents}
Let  $(U,\{e_{i}\}_{i\in I})$ and $(T,\{f_{j}\}_{j\in J})$ be basic w.e.i. $\K$-algebras, $M$ be an $U$-$T$-bimodule such that 
$\{e_iMf_j\}_{(i,j)\in I\times J}\subseteq\fl(\K)$  and lower triangular matrix $\K$-algebra $\Lambda=\left[\begin{smallmatrix}
T &  0 \\ 
M & U
\end{smallmatrix}\right]$ which is basic and with enough idempotents $\{g_r\}_{r\in I\vee J},$ see  Lemma \ref{LBasic1}. In this section, we find conditions under which $L=(Y,X,\varphi)\in \modu\,(\Lambda)$ implies that $Y\in \mathrm{mod}\,(T)$ and $X\in \mathrm{mod}\,(U)$. We also study modules of finite projective dimension in $\mathrm{mod}(\Lambda)$.
\

 Following \cite{LidiaJose}, we start this section by recalling the following notions. 

\begin{definition}\cite{LidiaJose} Let $(R,\{e_{i}\}_{i\in I})$ be a w.e.i. $\K$-algebra.  
\begin{itemize}
\item[(a)] $M\in\Mod\,(R)$ {\bf is endofinite} if $\{e_iM\}_{i\in I}\subseteq\fl\,(\End_R(M)^{op}).$  
\item[(b)] $R$ is left (respectively, right) {\bf locally endofinite} if $Re_{i}$ (respectively, $e_iR$) is endofinite, for every $i\in I.$ We say that $R$ is locally endofinite if it is both left and right locally endofinite.
\end{itemize}
\end{definition}

Let $(R,\{e_{i}\}_{i\in I})$ be a w.e.i. $\K$-algebra such that each $e_i$ is primitive. For $M\in\Mod(R),$ the support of $M$ is the set 
$\Supp(M):=\{i\in I\;:\; e_iM\neq 0\}.$ It is said that $R$ is left (respectively, right) {\bf support finite}  if   the set 
$\Supp(Re_i)$ (respectively $\Supp(e_i R)$) is finite for each $i\in I.$ We say that $R$ is support finite if it is both left and right support finite.

\begin{proposition}\cite[Proposition 7]{LidiaJose}\label{localbound}
Let $(R,\{e_{i}\}_{i\in I})$ be a w.e.i. $\K$-algebra such that each $e_i$ is primitive. Then, the following statements are equivalent.
\begin{enumerate}
\item [(a)] $R$ is right locally endofinite and left support finite.

\item [(b)] $\{Re_{i}\}_{i\in I}\subseteq\fl(R).$
\end{enumerate}
\end{proposition}
 
\begin{definition}\cite[Definition 9]{LidiaJose}
Let $(R,\{e_{i}\}_{i\in I})$ be a w.e.i. $\K$-algebra. It is said that $R$ is locally bounded if the following conditions hold true.
\begin{enumerate}
\item [(a)] each $e_i$ is primitive.
\item [(b)] $R$ is  locally endofinite. 
\item [(c)]  $R$ is support finite.
\end{enumerate}
\end{definition}

\begin{Remark}\label{oppLB} Let $(R,\{e_{i}\}_{i\in I})$ be a w.e.i. $\K$-algebra. Note that $(R,\{e_{i}\}_{i\in I})$ is locally bounded if, and only if, its opposite $\K$-algebra $(R^{op},\{e_{i}\}_{i\in I})$ is locally bounded.
\end{Remark}

\begin{proposition}\cite[Corollary 8, Proposition 10]{LidiaJose}\label{modRabelian}
For a locally bounded w.e.i. $\K$-algebra $(R,\{e_{i}\}_{i\in I}),$ the following statements hold true.
\begin{itemize}
\item[(a)] $\modu\,(R)=\fl(R)=\finp(R).$
\item[(b)] $\modu\,(R)$ is an abelian full subcategory of $\Mod\,(R)$ which is closed under the formation of  submodules.
\item[(c)] Let $M\in\Mod\,(R).$ Then, $M\in \modu\,(R)$ $\Leftrightarrow$ $\mathrm{Supp}(M)$ is finite and $e_iM\in\modu(e_iRe_i)$ $\forall\, i.$
\end{itemize}
\end{proposition}

\begin{lemma}\label{AuxElba}   Let $(R,\{e_{i}\}_{i\in I})$ be a w.e.i. $\K$-algebra and $N\in \Mod\,(R)$ such that $e_iN\in\fl(\K),$ for some $i\in I.$ Then $e_iN\in\fl(e_iRe_i).$
\end{lemma}
\begin{proof} Since $e_iRe_i$ is a $\K$-module, we get that $\varphi_i:K\to e_iRe_i,\, k\mapsto k\cdot e_i,$ is a morphism of rings 
with unity. In particular, by Proposition \ref{CRF}, the change of rings functor $(\varphi_i)_*:\Mod\,(e_iRe_i)\to \Mod\,(\K)$ is an exact 
$\K$-functor such that $(\varphi_i)_*(X)=X$ as abelian group $\forall\, X\in\Mod\,(e_iRe_i).$ Therefore any descending (ascending) 
chain of left   $e_iRe_i$-submodules of $N$ becomes stationary after finitely many steps and thus $e_iN\in\fl(e_iRe_i).$
\end{proof}

\begin{proposition}\label{Elba} Any support finite basic w.e.i. $\K$-algebra is locally bounded.
\end{proposition}
\begin{proof} Let $(R,\{e_{i}\}_{i\in I})$ be a support finite basic w.e.i. $\K$-algebra. We need to show that $R$ is locally endofinite. For each $i\in I,$ we have the ring isomorphism $\End_R(Re_i)\to e_iRe_i,\; f\mapsto f(e_i),$ that will we be used as an identification between these two rings. Then, by Lemma \ref{AuxElba} and its dual, we get that $e_jRe_i\in\fl(e_iR^{op}e_i)\cap\fl(e_jRe_j)$ since $e_jRe_i\in\fl(\K);$  proving the result. 
\end{proof}

\begin{Remark} A natural source of support finite  basic w.e.i. $\K$-algebras is given by the Bongartz-Gabriel construction \cite{BG}. In order to do that, we start by taking a locally bounded quiver $Q,$ a field $\K$ and an admissible ideal $I$ of the path $\K$-algebra $\K Q.$ Then, from the quotient path $\K$-algebra $\K Q/I,$ we can make a support finite basic w.e.i. $\K$-algebra, see all the details in \cite[Proposition 6.9]{MOSS}.
\end{Remark}

\begin{proposition}\label{Lblta} Let $(U,\{e_{i}\}_{i\in I})$ and $(T,\{f_{j}\}_{j\in J})$ be basic w.e.i. $\K$-algebras, $M$ be an 
$U$-$T$-bimodule such that 
$\{e_iMf_j\}_{(i,j)\in I\times J}\subseteq\fl(\K)$  and the lower triangular matrix $\K$-algebra 
$\Lambda=\left[\begin{smallmatrix} T &  0 \\ M & U \end{smallmatrix}\right]$ which is basic and with enough idempotents $\{g_r\}_{r\in I\vee J}$ (see  Lemma \ref{LBasic1}). Then, the following statements are equivalent.
\begin{itemize}
\item[(a)] $(\Lambda, \{g_r\}_{r\in I\vee J})$ is support finite.
\item[(b)] $(U,\{e_{i}\}_{i\in I})$ and $(T,\{f_{j}\}_{j\in J})$ are support finite, $\{Mf_j\}_{j\in J}\subseteq\modu\,(U)$ and 
$\{e_iM\}_{i\in I}\subseteq\modu\,(T^{op}).$
\item[(c)] $(U,\{e_{i}\}_{i\in I})$ and $(T,\{f_{j}\}_{j\in J})$ are support finite and the sets  $\Supp(Mf_j)$ and  $\Supp(e_iM)$ are finite $\forall\, i,j.$
\end{itemize}
Moreover, if one of the above equivalent conditions holds true, then $U,$ $T$ and $\Lambda$ are locally bounded.
\end{proposition} 
\begin{proof} By Lemma \ref{LBasicoTM} (c), for $i\in I$ and $j\in J,$ we get the equalities
 \begin{equation*}
 \begin{split}
 \Supp(\Lambda \overline{e}_i) & =\Supp(Ue_i)\times\{1\},\\
 \Supp(\overline{e}_i\Lambda ) & = (\Supp(e_iU)\cup\Supp(e_iM))\times\{1\},\\
  \Supp(\Lambda \overline{f}_j) & =(\Supp(Tf_j)\cup \Supp(Mf_j))\times\{2\},\\
  \Supp(\overline{f}_j\Lambda) & =\Supp(f_jT)\times\{2\}.
 \end{split}
 \end{equation*}
 On the other hand, by Lemma \ref{AuxElba} and its dual, we get that
 $$e_iMf_j\in\modu\,(e_iUe_i)\cap\modu\,(f_jT^{op}f_j)\;\forall\,i,j$$
 since $\{e_iMf_j\}_{(i,j)\in I\times J}\subseteq\fl(\K).$ Then, the result follows from Remark  \ref{oppLB},   Proposition \ref{modRabelian}  and Proposition \ref{Elba}.
\end{proof}

\begin{lemma}\label{geneproyec}
 For a w.e.i. $\K$-algebra $(R,\{e_{i}\}_{i\in I}),$ the following statements hold true.
\begin{enumerate}
\item [(a)]  Let $M\in \Mod\,(R).$ Then $M\in \modu\,(R)$ if,  and only if, there exists a finite subset $I'\subseteq I$ and an epimorphism $\coprod_{i\in I'}(Re_{i})^{\alpha_{i}}\to M$ of $R$-modules, where each $\alpha_i$ is a non-negative integer.

\item [(b)] $\proj(R)=\add\, (\{Re_{i}\}_{i\in I}).$ 
\end{enumerate}
\end{lemma}
\begin{proof} Note firstly that (b) follows from (a).  We recall from \cite[Section 49]{Wisbauer} that 
$(R,\{e_{i}\}_{i\in I})$ is a ring  with local units. That is, for a finitely many  $r_{1},\dots, r_{k}\in R$ there exists $e=e^2\in R,$ which is 
a finite sum of elements in $\{e_{i}\}_{i\in I},$  such that $\{r_j\}_{j=1}^k\subseteq eRe.$ 
\

Let $M\in\Mod\,(R).$ We assert that for any $m\in M$ there is $e=e^2\in R,$  which is 
a finite sum of elements in $\{e_{i}\}_{i\in I},$  such that $m=e m.$ Indeed, for $m\in M,$ there are $\{r_j\}_{j=1}^n\subseteq R$ and  
$\{m_j\}_{j=1}^n\subseteq M$ with $m=\sum_{j=1}^n r_jm_j.$ Since $R$ has local units, there exists $e=e^2\in R,$ which is 
a finite sum of elements in $\{e_{i}\}_{i\in I},$  such that $\{r_j\}_{j=1}^n\subseteq eRe.$ Then 
$m=\sum_{j=1}^n r_jm_j=\sum_{j=1}^n er_jm_j=em;$ proving our assertion.
\

Let us prove (a). Suppose that $M\in \modu\,(R).$ Then there is a finite generating set $\{m_j\}_{j=1}^k$ of the $R$-module $M.$ By 
the assertion above, for each $j,$ there exists $\varepsilon_j=\varepsilon_j^2=\sum_{t=1}^{m_j}\,e_{j_t}$ such that 
$m_j=\varepsilon_jm_j.$ Note that the mapping 
$$\coprod_{j=1}^k\;R\varepsilon_j\to M,\;(r_j\varepsilon_j)_{j=1}^k\mapsto\sum_{j=1}^k\,r_jm_j$$
is well defined, since $m_j=\varepsilon_jm_j$ for each $j,$ and it is an epimorphism of $R$-modules. Moreover, by using that 
$R\varepsilon_j=\oplus_{t=1}^{m_j}\,Re_{j_t}$ for each $j,$ we get the desired epimorphism of $R$-modules. 
\end{proof}

\begin{proposition}\label{Prop-fingen}  Let $(U,\{e_{i}\}_{i\in I})$ and $(T,\{f_{j}\}_{j\in J})$ be basic w.e.i. $\K$-algebras, $M$ be an 
$U$-$T$-bimodule such that 
$\{e_iMf_j\}_{(i,j)\in I\times J}\subseteq\fl(\K)$  and the lower triangular matrix $\K$-algebra 
$\Lambda=\left[\begin{smallmatrix} T &  0 \\ M & U \end{smallmatrix}\right]$ which is basic and with enough idempotents $\{g_r\}_{r\in I\vee J}$ (see  Lemma \ref{LBasic1}). Let  $(\Lambda, \{g_r\}_{r\in I\vee J})$ be support finite and $L:=(Y,\varphi,X)\in\modu\,(\Lambda).$ Then the following statements hold true.
\begin{itemize}
\item[(a)] There is an epimorphism of $\Lambda$-modules $\pi: P\to L,$ where $P=\coprod_{s\in S}\, (\Lambda g_s)^{\alpha_s}$ for some finite set $S\subseteq I\vee J$ and  non-negative integers $\alpha_s.$ 
\item[(b)] For $S_I:=S\cap (I\times\{1\})$ and $S_J:=S\cap (J\times\{2\}),$ we have
\begin{equation*}
\begin{split}
(i_T)_*(P)& \simeq \coprod_{s\in S_J}(Tf_{ p_2(s)})^{\alpha_s}\in\proj\,(T),\\
(i_U)_*(P)& \simeq \coprod_{s\in S_J}(Mf_{ p_2(s)})^{\alpha_s}\coprod\coprod_{s\in S_I}(U e_{ p_1(s)})^{\alpha_s}\in\modu\,(U).
\end{split}
\end{equation*}
\item[(c)]  $\pi:P\to L$ induces two epimorphisms $\pi_1:(i_T)_*(P)\to Y$ of $T$-modules and $\pi_2:(i_U)_*(P)\to X$ of $U$-modules. Moreover $Y\in\modu\,(T)$ and $X\in\modu\,(U).$
\item[(d)] Let $L\in\proj(\Lambda).$ Then, $Y\in\proj(T)$ and $X$ is a direct summand of $(i_U)_*(P).$
\end{itemize}
\end{proposition}
\begin{proof} By Proposition \ref{Lblta}, we know that $U,$ $T$ and $\Lambda$ are locally bounded. Moreover, $\{Mf_j\}_{j\in J}\subseteq\modu\,(U).$
\

Since $L$ is finitely generated, we get (a) from Lemma \ref{geneproyec} (a). Let now $\pi:P\to L$ be the epimorphism of $\Lambda$-modules from (a). Then, by Proposition \ref{MUmaps} (a), $P\simeq ((i_T)_*(P),\varphi',(i_U)_*(P)).$ In particular,  the epimorphism $\pi$ can be written as $\pi=(\pi_1,\pi_2):((i_T)_*(P),\varphi',(i_U)_*(P))\to (Y,\varphi,X)$ and thus, by Proposition \ref{MUmaps} (b), $\pi_1:(i_T)_*(P)\to Y$ is an epimorphism of $T$-modules and $\pi_2:(i_U)_*(P)\to X$ is an epimorphism of $U$-modules.
\

Let us compute $(i_T)_*(P)$ and $(i_U)_*(P).$ Indeed, from Remark \ref{MUmaps2} (4)  and since $(i_T)_*$ and $(i_U)_*$ are additive functors and $P=\coprod_{s\in S}\, (\Lambda g_s)^{\alpha_s},$ we get the isomorphisms stated in (b).  The fact that $(i_U)_*(P)\in\modu\,(U)$ follows from the inclusion $\{Mf_j\}_{j\in J}\subseteq\modu\,(U)$ and Proposition \ref{modRabelian} (b). Furthermore, 
$Y$ and $X$ are finitely generated since $\pi_1$ and $\pi_2$ are surjective. Finally, (d) follows from (b) and (c).
\end{proof}

\begin{lemma}\label{Lrespd} Let $0\to X_n\to \cdots\to X_1\to X_0\to A\to 0$ be an exact sequence in an abelian category $\A.$ Then 
$\pd(A)\leq n+\max\{\pd(X_i)\}_{i=1}^n.$
\end{lemma}
\begin{proof} By using that $\max\{a,b+c\}\leq\max\{a,b\}+c$ for any non-negative integers $a,b$ and $c,$ the proof can be carry on by induction on 
$n.$
\end{proof}

\begin{proposition}\label{PropPD}   Let $(U,\{e_{i}\}_{i\in I})$ and $(T,\{f_{j}\}_{j\in J})$ be basic w.e.i. $\K$-algebras, $M$ be an 
$U$-$T$-bimodule such that 
$\{e_iMf_j\}_{(i,j)\in I\times J}\subseteq\fl(\K)$  and the lower triangular matrix $\K$-algebra 
$\Lambda=\left[\begin{smallmatrix} T &  0 \\ M & U \end{smallmatrix}\right]$ which is basic and with enough idempotents $\{g_r\}_{r\in I\vee J}$ (see  Lemma \ref{LBasic1}). If  $(\Lambda, \{g_r\}_{r\in I\vee J})$ is support finite, then the following statements hold true.
\begin{itemize}
\item[(a)] For any $X\in\Mod\,(U),$ $(p_U)_*(X)\in \modu\,(\Lambda)\;\Leftrightarrow\;X\in\modu\,(U).$ 
\item[(b)]  For any $Y\in\Mod\,(T),$ $(p_T)_*(Y)\in \modu\,(\Lambda)\;\Leftrightarrow\;Y\in\modu\,(T).$
\item[(c)]  For any $X\in\modu\,(U),$ there is a finite set $J_X\subseteq J$ such that 
$$\pd((p_U)_*(X))\leq \pd(X)\leq \pd((p_U)_*(X))+\max\{\pd(Mf_j)\}_{j\in J_X}.$$
\item[(d)]  For any $Y\in\modu\,(T),$ there is a finite set $J_Y\subseteq J$ such that 
$$\pd(Y)\leq \pd((p_T)_*(Y))\leq\pd(Y)+1+\max\{\pd(Mf_j)\}_{j\in J_Y}.$$
\end{itemize}
\end{proposition}
\begin{proof} Let $(\Lambda, \{g_r\}_{r\in I\vee J})$ be support finite. Consider $\mathsf{E}:=\add\,(\{Mf_j\}_{j\in J}).$
\

(a) Let $X\in\Mod\,(U).$ If $(p_U)_*(X)\in \modu\,(\Lambda),$  by Proposition \ref{Prop-fingen} we have that
$X\in \modu\,(U).$ Suppose that $X\in \modu\,(U).$ Then, there are finitely many $x_1,x_2,\cdots,x_n$ in $X$ such that 
$X=\sum_{i=1}^n\, Ux_i.$ On the other hand, by Remark \ref{MUmaps2} (3), $(p_U)_*(X)$ can be seen as the matrix 
$\left[\begin{smallmatrix} 0 &  0 \\ 0 & X \end{smallmatrix}\right]$ and thus $\sum_{i=1}^n\, \Lambda \left[\begin{smallmatrix} 0 &  0 \\ 0 & x_i \end{smallmatrix}\right]=\left[\begin{smallmatrix} 0 &  0 \\ 0 & \sum_{i=1}^n\, Ux_i \end{smallmatrix}\right]=\left[\begin{smallmatrix} 0 &  0 \\ 0 & X \end{smallmatrix}\right];$ proving that $(p_U)_*(X)\in \modu\,(\Lambda).$
\

(b) It can be proved similarly as we did in (a).
\

(c) Let $X\in \modu\,(U).$ Assume that $n:=\pd((p_U)_*(X))<\infty.$ Then, by Proposition \ref{Prop-fingen} and Proposition \ref{MUmaps} (b), there is an exact sequence 
$$0\to E_n\coprod Q_n\to \cdots \to E_1\coprod Q_1\to E_0\coprod Q_0\to X\to 0$$
 in $\modu\,(U),$ where $E_i\in\mathsf{E}$ and $Q_i\in\proj(U)$ for all $i.$ In particular, we can form a finite set $J_X\subseteq J$ 
 such that 
  $\{E_i\}_{i=1}^n\subseteq \add\,(\coprod_{j\in J_X}\,Mf_j).$ Therefore, from Lemma \ref{Lrespd}, and the preceding exact sequence, we get that 
 $\pd(X)\leq \pd((p_U)_*(X))+\max\{\pd(Mf_j)\}_{j\in J_X}.$
\

Assume now that $m:=\pd\,(X)<\infty.$ Then, there is a projective resolution $\eta:\;0\to Q_m\to \cdots\to Q_1\to Q_0\to X\to 0$ 
of $X.$ Since $(p_U)_*$ is an exact functor and $(p_U)_*(Ue_i)\simeq \Lambda\overline{e}_i\in\proj(\Lambda),$ by applying the functor 
$(p_U)_*$ to $\eta,$ we get a projective resolution (of length $m$) of $(p_U)_*(X)$ and hence $\pd((p_U)_*(X))\leq \pd\,(X).$
\

(d)  Let $Y\in \modu\,(T).$ Suppose there is a projective resolution (of length $m$) of $(p_T)_*(Y).$ Then, by Proposition \ref{Prop-fingen} and Proposition \ref{MUmaps} (b), there is a projective resolution (of length $m$) of $Y$ and thus 
$\pd(Y)\leq \pd((p_T)_*(Y)).$
\

Assume that $n:=\pd(Y)<\infty.$ Then, there is a projective resolution 
$0\to Q_n\to\cdots\to Q_1\to Q_0\to Y\to 0$ of $Y.$ Thus, by applying the exact functor $(p_T)_*$ to this resolution, we get the exact sequence of $\Lambda$-modules
$$\theta:\;0\to (p_T)_*(Q_n)\to\cdots\to (p_T)_*(Q_1)\to (p_T)_*(Q_0)\to (p_T)_*(Y)\to 0.$$
By Lemma \ref{geneproyec}, $Q_i\in\proj(T)=\add\,(\{Tf_j\}_{j\in J}).$ Thus, we can form a finite set $J_Y\subseteq J$ 
 such that $\{(p_T)_*(Q_i)\}_{i=1}^n\subseteq \add\,(\{(p_T)_*(Tf_j)\}_{j\in J_Y}.$ Therefore, 
 $\pd\,((p_T)_*(Q_i))\leq \max\{\pd((p_T)_*(Tf_j))\}_{j\in J_Y},$ for all $i.$
 \
 
 Let $j\in J.$ Note that $\Lambda\overline{f}_j\simeq(Tf_j,\varphi,Mf_j))$ and thus we have the exact sequence $0\to (p_U)_*(Mf_j)\to \Lambda\overline{f}_i\to (p_T)_*(Tf_j)\to 0$ of $\Lambda$-modules. Then, by (c) and the preceding exact sequence,  we get
 $$\pd((p_T)_*(Tf_j))\leq 1+\pd((p_U)_*(Mf_j))\leq 1+\pd(Mf_j).$$
Therefore $\max\{\pd\,((p_T)_*(Q_i))\}_{i=1}^n\leq 1+\max\{\pd(Mf_j)\}_{j\in J_Y}.$ Hence to finish the proof, it is enough to apply Lemma \ref{Lrespd} to the exact sequence $\theta.$
\end{proof}

\begin{theorem}\label{MainT2}  For $(U,\{e_{i}\}_{i\in I})$ and $(T,\{f_{j}\}_{j\in J})$  basic w.e.i. $\K$-algebras, $M$ an 
$U$-$T$-bimodule such that 
$\{e_iMf_j\}_{(i,j)\in I\times J}\subseteq\fl(\K)$  and the lower triangular matrix $\K$-algebra 
$\Lambda=\left[\begin{smallmatrix} T &  0 \\ M & U \end{smallmatrix}\right],$ which is basic and with enough idempotents $\{g_r\}_{r\in I\vee J}$ (see  Lemma \ref{LBasic1}), such that  $(\Lambda, \{g_r\}_{r\in I\vee J})$ is support finite, the following statements hold true.
\begin{itemize}
\item[(a)] Let $L=(Y,\varphi,X)\in \Mod\,(\Lambda)$ and $\pd(Mf_j)<\infty$ $\forall\,j\in J.$ Then
$$L\in\mathcal{P}^{<\infty}_\Lambda\;\Leftrightarrow\; Y\in\mathcal{P}^{<\infty}_T\;\text{ and }\,X\in\mathcal{P}^{<\infty}_U.$$
\item[(b)] Let $m:=\max\,\{\pd(Mf_j)\}_{j\in J}<\infty,$ $a:=\findim(T),$ $\alpha:=\gldim(T),$ $b:=\findim(U)$ and $\beta:=\gldim(U).$ Then
\begin{equation*}
\begin{split}
\max\,\{\beta-m,\alpha\}\leq\gldim(\Lambda) & \leq\max\,\{\beta,\alpha+1+m\},\\
\max\,\{b-m,a\}\leq\findim(\Lambda) & \leq\max\,\{b,a+1+m\}.
\end{split}
\end{equation*}
\end{itemize}
\end{theorem}
\begin{proof} (a) Consider the exact sequence $0\to (p_U)_*(X)\to L\to (p_T)_*(Y)\to 0$ of $\Lambda$-modules. Then, by Proposition \ref{PropPD}, we get that $L\in\modu\,(\Lambda)$ if, and only if, $Y\in\modu\,(T)$ and $X\in\modu\,(U).$
\

If $Y\in\mathcal{P}^{<\infty}_T$ and $X\in\mathcal{P}^{<\infty}_U,$ then by Proposition \ref{PropPD} and the above exact sequence, we conclude that $L\in\mathcal{P}^{<\infty}_\Lambda.$
\

Let $L\in\mathcal{P}^{<\infty}_\Lambda$ and $n:=\pd\,(L).$ Then by Proposition \ref{Prop-fingen} and Proposition \ref{MUmaps} (b), there are two exact sequences
$$\theta_Y\;:0\to P_n\to \cdots\to P_1\to P_0\to Y\to 0$$
of $T$-modules, where $P_i\in\proj(T)$ $\forall\, i;$ and
$$\theta_X\;:\;0\to E_n\coprod Q_n\to \cdots \to E_1\coprod Q_1\to E_0\coprod Q_0\to X\to 0$$
 in $\modu\,(U),$ where $E_i\in\add\,(\{Mf_j\}_{j\in J})$ and $Q_i\in\proj(U)$ for all $i.$ In particular, we can form a finite set $J_X\subseteq J$ 
 such that 
  $\{E_i\}_{i=1}^n\subseteq \add\,(\coprod_{j\in J_X}\,Mf_j).$
Now, from the exact sequence $\theta_Y,$ we get that $\pd(Y)\leq n<\infty.$ On the other hand, by applying Lemma \ref{Lrespd} to the exact sequence $\theta_X,$ it follows that 
$\pd\,(X)\leq n+\max\,\{\pd(Mf_j)\}_{j\in J_X}<\infty.$
\

(b) Let $L=(Y,\varphi,X)\in \modu\,(\Lambda).$ Then, from the exact sequence 
$0\to (p_U)_*(X)\to L\to (p_T)_*(Y)\to 0$ of $\Lambda$-modules and Proposition \ref{PropPD} (c,d)
\begin{equation*}
\begin{split}
\pd(L) & \leq\max\,\{\pd((p_U)_*(X)),\pd( (p_T)_*(Y))\}\\
& \leq\max\,\{\pd(X),\pd(Y)+1+\max\{\pd(Mf_j)\}_{j\in J_Y}\}\\
& \leq\max\,\{\gldim(U),\gldim(T)+1+m\},
\end{split}
\end{equation*}
and thus $\gldim(\Lambda)\leq \max\,\{\gldim(U),\gldim(T)+1+m\}.$
\

Let $Y\in\modu\,(T).$ Then, by Proposition \ref{PropPD} (d), 
$\pd(Y)\leq \pd((p_T)_*(Y))\leq \gldim(\Lambda)$ and thus $\gldim(T)\leq \gldim(\Lambda).$
\

Let $X\in\modu\,(U).$ Then, by Proposition \ref{PropPD} (d),
$$\pd(X)\leq \pd((p_U)_*(X))+\max\{\pd(Mf_j)\}_{j\in J_X}\leq \gldim(\Lambda)+m$$
and thus $\gldim(U)\leq \gldim(\Lambda)+m.$ Finally, the proof of the inequalities involving the finitistic dimension can be done, by using (a), as in the one we did for the global dimension.
\end{proof}

\begin{theorem}\label{MainT3}  Let $(U,\{e_{i}\}_{i\in I})$ and $(T,\{f_{j}\}_{j\in J})$  be basic w.e.i. $\K$-algebras, $M$ be an 
$U$-$T$-bimodule such that 
$\{e_iMf_j\}_{(i,j)\in I\times J}\subseteq\fl(\K)$ and $\pd(Mf_j)<\infty$ $\forall\,j\in J.$  Let $\Lambda:=\left[\begin{smallmatrix} T &  0 \\ M & U \end{smallmatrix}\right],$ which is basic and with enough idempotents $\{g_r\}_{r\in I\vee J}$ (see  Lemma \ref{LBasic1}), be such that  $(\Lambda, \{g_r\}_{r\in I\vee J})$ is support finite. Consider the partitions
$\tilde{\A}=\{\tilde{\A}_i\}_{i<\alpha}\in\wp(I),$  
$\tilde{\B}=\{\tilde{\B}_j\}_{i<\beta}\in\wp(J),$ and
$\tilde{\C}:=\tilde{\A}\vee\tilde{\B}\in\wp(I\vee J)$  (see \ref{JSLema}). Then, the following statements are equivalent.
\begin{itemize}
\item[(a)] $\F({}_{\tilde{\C}}\Delta)= \mathcal{P}_{\Lambda}^{<\infty}.$
\item[(b)] $\F({}_{\tilde{\A}}\Delta)= \mathcal{P}_{U}^{<\infty}$ and $\F({}_{\tilde{\B}}\Delta)= \mathcal{P}_{T}^{<\infty}.$
\end{itemize}
Moreover, if one of the above equivalent conditions holds true, then $\Lambda,$ $U$ and $T$ are locally bounded and left standardly stratified $\K$-algebras, $\findim(\Lambda)=\pd({}_{\tilde{\C}}\Delta),$ $\findim(T)=\pd({}_{\tilde{\B}}\Delta)$ and 
$\findim(U)=\pd({}_{\tilde{\A}}\Delta).$
\end{theorem}
\begin{proof} From Lemma \ref{LBasic2} (a,b) and Theorem \ref{MainT2} (a), it can be shown that
$$(*)\quad {}_{\tilde{\C}}\Delta\subseteq \mathcal{P}_{\Lambda}^{<\infty}\;\Leftrightarrow\;{}_{\tilde{\A}}\Delta\subseteq \mathcal{P}_{U}^{<\infty}\;\text{ and }\;{}_{\tilde{\B}}\Delta\subseteq \mathcal{P}_{T}^{<\infty}.$$

(a) $\Rightarrow$ (b): Since the classes $\mathcal{P}_{U}^{<\infty}$ and $\mathcal{P}_{T}^{<\infty}$ are closed under extensions, we get from (*), that 
$\F({}_{\tilde{\A}}\Delta)\subseteq \mathcal{P}_{U}^{<\infty}$ and $\F({}_{\tilde{\B}}\Delta)\subseteq\mathcal{P}_{T}^{<\infty}.$
\

Let $X\in\mathcal{P}_{U}^{<\infty}.$ Then, by Theorem \ref{MainT2} (a), 
$(0,0,X)\in \mathcal{P}_{\Lambda}^{<\infty}=\F({}_{\tilde{\C}}\Delta).$ Hence by Lemma \ref{LBasic2} (e), we conclude that $X\in\F({}_{\tilde{\A}}\Delta)$ and thus 
$\mathcal{P}_{U}^{<\infty}\subseteq \F({}_{\tilde{\A}}\Delta).$ Similarly, it can be shown that $\mathcal{P}_{T}^{<\infty}\subseteq \F({}_{\tilde{\B}}\Delta).$
\

(b) $\Rightarrow$ (a): Since the classes $\mathcal{P}_{\Lambda}^{<\infty}$ is closed under extensions, we get from (*), that 
$\F({}_{\tilde{\Lambda}}\Delta)\subseteq \mathcal{P}_{\Lambda}^{<\infty}.$ Let 
$L=(Y,\varphi,X)\in \mathcal{P}_{\Lambda}^{<\infty}.$ Then, by Theorem \ref{MainT2} (a), $Y\in\mathcal{P}_{T}^{<\infty}=\F({}_{\tilde{\B}}\Delta)$ and 
$X\in\mathcal{P}_{U}^{<\infty}=\F({}_{\tilde{\A}}\Delta).$ Thus, by Lemma \ref{LBasic2} (e), $L\in\F({}_{\tilde{\C}}\Delta).$ 
\

Assume now that one of the above equivalent conditions holds true. Then, by Proposition \ref{CriterioSS} and Proposition \ref{Lblta}, we get that $\Lambda,$ $U$ and $T$ are locally bounded standardly stratified $\K$-algebras. Moreover, 
$\findim(\Lambda)=\pd(\mathcal{P}_{\Lambda}^{<\infty})=\pd(\F({}_{\tilde{\Lambda}}\Delta))=\pd({}_{\tilde{\Lambda}}\Delta).$ Finally, in a similar fashion, we can compute $\findim(U)$ and $\findim(T).$
\end{proof}

\footnotesize

\vskip3mm \noindent Eduardo Marcos:\\
Instituto de Matem\'aticas y Estadistica,\\
Universidad de Sao Paulo,\\
Sao Paulo, BRASIL.

{\tt enmarcos@ime.usp.br}

\vskip3mm \noindent Octavio Mendoza:\\
Instituto de Matem\'aticas,\\
Universidad Nacional Aut\'onoma de M\'exico,\\
Circuito Exterior, Ciudad Universitaria,\\
M\'exico D.F. 04510, M\'EXICO.

{\tt omendoza@matem.unam.mx}

\vskip3mm \noindent Corina S\'aenz:\\
Departamento de Matem\'aticas, Facultad de Ciencias,\\
Universidad Nacional Aut\'onoma de M\'exico,\\
Circuito Exterior, Ciudad Universitaria,\\
M\'exico D.F. 04510, M\'EXICO.

{\tt corina.saenz@gmail.com}

\vskip3mm \noindent Valente Santiago:\\
Departamento de Matem\'aticas, Facultad de Ciencias,\\
Universidad Nacional Aut\'onoma de M\'exico,\\
Circuito Exterior, Ciudad Universitaria,\\
M\'exico D.F. 04510, M\'EXICO.

{\tt valente.santiago.v@gmail.com}

\end{document}